%% file: VerlindeSums4.tex
\documentclass[11pt]{amsart}
\usepackage{amsmath,amscd,amssymb,amsfonts,amsthm,bm,hyperref}
\usepackage{xcolor}  
\hypersetup{
    colorlinks,
    linkcolor={red!50!black},
    citecolor={blue!70!black},
    urlcolor={blue!80!black}
}
\usepackage[cmtip,matrix,arrow]{xy}

\newtheorem{theorem}{Theorem}[section]

\newtheorem{proposition}[theorem]{Proposition}

\newtheorem{lemma}[theorem]{Lemma}
\theoremstyle{definition}
\newtheorem{definition}[theorem]{Definition}
\theoremstyle{remark}

\newtheorem{remark}[theorem]{Remark}

\newtheorem{example}[theorem]{Example}

\newcommand\A{\mathcal{A}}

\newcommand{\ca}{\mathcal}

\newcommand{\U}{\on{U}}

\newcommand{\R}{\mathbb{R}}
\newcommand{\C}{\mathbb{C}}
\newcommand{\SU}{\on{SU}}

\newcommand{\Z}{\mathbb{Z}}

\renewcommand{\P}{\mathsf{P}}

\newcommand\lie[1]{\mathfrak{#1}}

\newcommand{\h}{\lie{h}}

\renewcommand{\t}{\lie{t}}

\newcommand{\on}{\operatorname}
\newcommand{\Hom}{ \on{Hom}}

\renewcommand{\subset}{\subseteq}
\renewcommand{\supset}{\supseteq}
\renewcommand{\ker}{ \on{ker}}

\newcommand\qu{/\kern-.7ex/} 

\renewcommand{\d}{{\mbox{d}}}
\newcommand{\ol}{\overline}

\newcommand\eps{\epsilon}
\newcommand\Om{\Omega}

\newcommand{\f}{\frac}

\renewcommand{\l}{\langle}
\renewcommand{\r}{\rangle}
\newcommand\hh{{\f{1}{2}}}

\newcommand{\eeq}{\end{eqnarray*}}
\newcommand{\beq}{\begin{eqnarray*}}

\newcommand{\End}{\on{End}}

\newcommand{\mf}{\mathfrak}
\newcommand{\ann}{\on{ann}}

\newcommand\lieg{{\mathfrak{g}}}
\newcommand\lieh{{\mathfrak{h}}}
\newcommand\liet{{\mathfrak{t}}}

\newcommand{\lieT}{{\mathsf{T}}}
\newcommand{\lieH}{{\mathsf{H}}}

\newcommand\Eul{{\textnormal{Eul}}}

\newcommand\Pol{{\textnormal{Pol}}}
\newcommand\vol{{\textnormal{vol}}}

\newcommand\Ber{{\textnormal{Ber}}}

\newcommand\Map{{\textnormal{Map}}}
\renewcommand\sp{{\textnormal{span}}}
\newcommand{\ul}{\underline}

\newcommand\vx{{\textnormal{vx}}}
\newcommand\Ver{{\textnormal{Ver}}}

\newcommand\Sym{{\textnormal{Sym}}}

\newcommand\frakc{{\mathfrak{c}}}

\newcommand\calR{{\mathcal{R}}}

\newcommand\calA{{\mathcal{A}}}
\newcommand\calD{{\mathcal{D}}}

\newcommand\calS{{\mathcal{S}}}

\newcommand\ualpha{{\underline{\alpha}}}

\newcommand\hvee{{\textnormal{h}^\vee}}
\newcommand\cox{\hvee}

\newcommand\bC{{\mathbb{C}}}
\newcommand\bR{{\mathbb{R}}}
\newcommand\bQ{{\mathbb{Q}}}
\newcommand\bZ{{\mathbb{Z}}}

\newcommand\bN{{\mathbb{N}}}

\newcommand\bfalpha{{\bm{\alpha}}}
\newcommand\ubfalpha{{\ul{\bfalpha}}}
\newcommand\bfu{{\bm{u}}}
\newcommand\bfbeta{{\bm{\beta}}}

\newcommand\bfH{{\bm{H}}}

\newcommand\sm{{\setminus}}

\renewcommand{\i}{{\mathrm{i}}}

\newcommand\calDC{{\mathcal{D}_{\mathbb{C}}}}
\newcommand\Td{{\textnormal{Td}}}

\newcommand{\pair}[2]{\langle #1, #2 \rangle}
\newcommand{\ignore}[1]{}
\DeclareMathOperator*{\CT}{CT}  
\DeclareMathOperator*{\iCT}{iCT}

\begin{document}
	\sloppy
	\title{The decomposition formula for Verlinde Sums}
	\author{Yiannis Loizides}
	
	\author{Eckhard Meinrenken}
	
\begin{abstract}
We prove a decomposition formula for Verlinde sums (rational trigonometric sums), as a discrete counterpart to the Boysal-Vergne decomposition formula for Bernoulli series. Motivated by applications to fixed point formulas in Hamiltonian geometry, we develop differential form valued version of Bernoulli series and Verlinde sums, and extend the decomposition formula to this wider context. 
\end{abstract}
\maketitle

\section{Introduction}
The prototypical example of a Verlinde sum is the function, for a given natural number $n$,
\begin{equation}\label{eq:versum}
 V_n(\lambda,\ell)=\sum_{q^\ell=1,\ q\neq 1} \frac{q^\lambda}{(1-q^{-1})^n},\ \ \ 
 (\lambda,\ell)\in \bZ\times \bN. \end{equation}
It may be seen as a discrete analogue of the Bernoulli series
\begin{equation}\label{eq:berseries}
 B_n(\lambda)=\sum_{j\in \bZ_{\neq 0}} \frac{e^{2\pi \i j\lambda}}{(2\pi \i j)^n},\ \ \ \ 
  \lambda\in \bR.
 \end{equation}
The Bernoulli series is invariant under integer translation, and its restriction to the open unit interval $(0,1)$ is given by a polynomial $\on{Ber}_n$ of degree $n$ in $\lambda$. 
For $n\ge 1$, these are $-\f{1}{n!}$ times the  \emph{Bernoulli polynomials}. Similarly, the
Verlinde sum for fixed $\ell$ is $\ell\bZ$-periodic in $\lambda$, and its restriction 
to the set of all $(\lambda,\ell)$ such that $\lambda/\ell\in (0,1)$  is a polynomial $\on{Ver}_n$ of degree $n$ in the two variables $\lambda,\ell$. Its part of homogeneity $n$ in $(\lambda,\ell)$ is given by  $\ell^n \on{Ber}_n(\lambda/\ell)$.
These polynomials may be described through generating functions:
\begin{equation}\label{eq:vergenfunction}
 \sum_{n=0}^\infty (1-e^{-z})^n \on{Ver}_n(\lambda,\ell)
=1-\ell e^{\lambda z} \f{e^z-1}{e^{\ell z}-1}
\end{equation}
and 
\begin{equation}\label{eq:bergenfunction}
 \sum_{n=0}^\infty z^n \on{Ber}_n(\lambda)
=1-e^{\lambda z} \f{z}{e^{z}-1}.
\end{equation}

 The Bernoulli series make an appearance in the volume formula for moduli spaces of flat connections of flat $\on{SU}(2)$-bundles over surfaces, due to Thaddeus \cite{th:co} and  Witten \cite{wi:qg}.  Similarly, the Verlinde sums appear in the Verlinde formulas \cite{ve:fr} for the quantization of these moduli spaces.  

The versions for  higher rank Lie groups lead to higher-dimensional \emph{multiple} Bernoulli series and Verlinde sums.  The multiple Bernoulli series exhibit a piecewise polynomial behavior, while the multiple Verlinde sums are piecewise \emph{quasi}-polynomial. Aside from the context of 2-dimensional gauge theory, they 
also appear in the localization formulas for volumes and quantizations of Hamiltonian loop group spaces or, equivalently, quasi-Hamiltonian $G$-spaces \cite{al:fi,al:du,me:twi}.

The combinatorics of the (multiple) Verlinde sums, also known as \emph{rational trigonometric sums}, was developed by Szenes \cite{sz:res}. The main result in \cite{sz:res} is a residue formula for the Verlinde sums,  similar to his earlier result \cite{sz:it} for Bernoulli series. The Szenes formula allows for efficient  computations, and in particular leads to a proof of the (quasi-)polynomial behavior -- for example, 
\eqref{eq:vergenfunction} and \eqref{eq:bergenfunction} may be obtained from his theorem. 

The main purpose of this article is to prove a \emph{decomposition formula} for multiple Verlinde sums, similar to the Boysal-Vergne decomposition formula \cite{boy:mul} for Bernoulli series. For the Verlinde sums $V_n$, the formula is 
\[ V_n(\lambda,\ell)=\Ver_n(\lambda,\ell)+\ell\sum_{\mu=1}^\infty P_n(\lambda-\ell\mu)+(-1)^n \ell \sum_{\mu=-\infty}^0 P_n(\ell\mu-\lambda-n),
\]
where $P_n(\mu)$ is the partition function, i.e., the number of ways of writing $\mu$ as a sum of $n$ non-negative integers. It expresses $V_n$ as a central polynomial contribution,  plus correction terms supported on affine half lines; the support properties are such that for any given $(\lambda,\ell)$ the sum is finite.

We will need the decomposition formula for Verlinde sums for our combinatorial proof of the `quantization commutes with reduction' theorem for Hamiltonian loop group spaces, in the spirit of Szenes-Vergne's argument \cite{sz:qr0} for ordinary Hamiltonian spaces. In turn, this approach is motivated by Paradan's norm-square formulas for Hamiltonian spaces \cite{par:loc,par:wa}.  With these applications in mind, we are led to consider more general Verlinde sums with `equivariant parameters';  under the Chern-Weil homomorphism this produces the required formulas for expressions arising in the fixed point formula.

\bigskip\noindent 
\noindent{\bf Acknowledgments:} We  thank Michele Vergne for discussions and helpful comments. Most of the results described in this paper are natural extensions of work by Arzu Boysal and Michele Vergne on Bernoulli series, and a number of steps in this direction are already explained in Vergne's lecture notes \cite{ve:res}; as well as the slides \cite{ver:sl} from her lecture at the 2010 AMS meeting in San Francisco. 
\bigskip

\section{Verlinde sums}

\subsection{Lattices.}
Let $\Lambda$ be a lattice, with dual lattice $\Lambda^\ast=\Hom(\Lambda,\bZ)$. 
We denote by $\liet=\Lambda\otimes_{\bZ}\bR$ the real vector space spanned by 
$\Lambda$ and by  $\liet_{\bQ}=\Lambda\otimes_{\bZ}\bQ$ its rational points. 
A subspace $\lieh \subset \liet$ is \emph{rational}  if it is spanned by $\lieh\cap \liet_{\bQ}$; in this case, $\ann(\lieh)\subset \liet^\ast$ is rational and 
$(\lieh\cap \Lambda)^\ast\cong \Lambda^\ast/\Lambda^\ast\cap \ann(\lieh)$. 

Let $\lieT=\liet/\Lambda$ the torus; the quotient map is the exponential map $\exp\colon \liet\to \lieT$ for this torus. Thus $\Lambda=\ker(\exp)$ becomes the integral lattice of $\lieT$, and $\Lambda^\ast$ is identified with the weight lattice $\mathrm{Hom}(T,\mathrm{U}(1))$.  For $\lambda\in\Lambda^\ast$, we denote by $\lieT\to \mathrm{U}(1),\ \ t\mapsto t^\lambda$ the corresponding homomorphism; thus 
\begin{equation}\label{eq:weights} t^\lambda=e^{2\pi \i \pair{\lambda}{X}}\end{equation}
for $t=\exp(X)$.  Suppose $\Xi\subset \liet$ is a lattice containing $\Lambda$.  For $\ell\in \bN$ we consider the finite subgroup
\[ \lieT_\ell=\ell^{-1}\Xi/\Lambda \subset \lieT.\]
For example, if $B$ is an inner product on $\liet$ which is \emph{integral} in the sense that it restricts to a $\bZ$-valued bilinear form on $\Lambda$, then the inverse image 
of $\Lambda^*$ under the  isomorphism $B^\flat\colon \liet\to \liet^\ast$ can play the role of such a lattice $\Xi$.

\begin{example}\label{ex:lie1}The following setting plays a role for the Verlinde formula: 
	Let $G$ be a compact, simple, simply connected Lie group, $\lieT\subset G$ a maximal torus, 
	and $\Lambda\subset \liet$ the integral lattice, so that $\lieT=\liet/\Lambda$. The \emph{basic inner product} $B$ on $\lieg$ is the unique invariant inner product 
	such that the shortest vectors in $\Lambda-\{0\}$ have length $\sqrt{2}$. This is an integral inner product, and we take $\Xi=\Lambda^*$ under the resulting 
	identification $\liet\cong \liet^\ast$. Given a level $k\in \bN$, the \emph{level $k$ fusion ring} (or \emph{Verlinde algebra}) $R_k(G)$ can be abstractly defined as 
	quotient ring  $R(G)/\mathcal{I}_k(G)$, where $R(G)$ is regarded as the ring of characters of $G$-representations, and $\mathcal{I}_k(G)$ is the ideal of characters vanishing on the regular elements of 
	\[ \lieT_{k+\cox}\subset \lieT .\] 
	Here $\cox$ is the dual Coxeter number of $G$, and an element of $G$ is called regular if its centralizer is a maximal torus.
\end{example}

Recall that a \emph{list} $A$ of elements of a set $S$ is a collection of elements of $S$ `with multiplicities'. Equivalently, it amounts to a map $S\to \bN\cup\{0\}$. For finite lists, it is often convenient to  use set-theoretic notation $A=\{s_1,\ldots,s_n\}$,
where elements appear several times, according to their multiplicity. The notation $s\in A$ means that $s$ appears with multiplicity at least $1$, and $A-\{s\}$ is the new list in which the multiplicity of $s$ has been reduced by $1$.
The Verlinde sums in the following section will involve a choice of list $\bfalpha=\{\alpha_1,\ldots,\alpha_n\}$ of weights $\alpha_i\in\Lambda^*$. Note that
finite lists of weights in $\Lambda^*$ are in 1-1 correspondence with isomorphism classes of finite-dimensional unitary $\lieT$-representations.

\begin{definition}\label{def:arrangements}
For a list $\bfalpha=\{\alpha_1,\ldots,\alpha_n\}$ of weights,  we define:
  \begin{itemize}
\item 
The  \emph{zonotope} 
\[ \Box \bfalpha = \Big\{ \sum_{k=1}^n t_k \alpha_k\in\liet^\ast| \ 0 \le t_k \le 1 \Big\}.\]
\item 
The collection $\calS=\calS(\bfalpha)$ of affine subspaces of $\liet^\ast$, consisting of  subspaces spanned by sublists  of $\bfalpha$, together with $\Xi^\ast$-translates of such subspaces. The elements $\Delta\in \calS$ are referred to as \emph{admissible subspaces} \cite{boy:mul}. 
\item For $\Delta\in\calS(\bfalpha)$, we denote by $\liet_\Delta\subset \liet$ the rational subspace of vectors orthogonal to $\Delta$,  and by $\lieT_\Delta=\exp(\liet_\Delta)$ the corresponding torus. If $\Delta$ is a 
$\Xi^\ast$-translate of the span of some $\alpha\in \bfalpha$, these coincide with $\liet_\alpha=\ker(\alpha),\ \lieT_\alpha=\exp(\liet_\alpha)$.
\item Given $\Delta\in \calS(\bfalpha)$, an element $\mu\in \Delta$ is called \emph{regular in $\Delta$} if it is not contained in any $\Delta'\in \calS(\bfalpha)$
with $\dim\Delta'<\dim\Delta$. The connected components of regular elements in $\Delta$ are called the (open) \emph{chambers in $\Delta$}. 
If $\Delta_0:=\sp_{\bR}\,\bfalpha$ is all of $\liet^\ast$, the chambers of  
$\Delta_0\in\calS(\bfalpha)$ will simply be referred to as the \emph{chambers}. 
\end{itemize}
\end{definition}

\begin{example}\label{ex:lie2}
In Example \ref{ex:lie1}, consider the list $\bfalpha=\mathfrak{R}_-$ 
of negative roots of $G$. 
The collection $\calS$ of admissible subspaces consists of subspaces spanned by subsets of roots, together with their $\Xi^\ast\cong \Lambda$-translates.  Letting $\rho$ be the 
half-sum of positive roots, the zonotope is given by 
\[ \Box\mathfrak{R}_-=\mathrm{hull}(W\cdot\rho)-\rho,\]
 the $-\rho$ shift of the convex hull of the Weyl group orbit of $\rho$. 
This follows  from the formula for the character of the $\rho$-representation, 
\[ \chi_\rho(t)=t^\rho \prod_{\alpha\in\mathfrak{R}_-}(1+t^\alpha)\]
by comparing the set of weights appearing on both sides. 
\end{example}

\subsection{Definition of the Verlinde sum.} \label{SubsectionDefandQuasi}
Suppose $\Lambda \subset \Xi$ are full-rank lattices in a  finite-dimensional vector space $\liet$ of dimension $r$, and recall $\lieT_\ell=(\frac{1}{\ell}\Xi)/\Lambda$.  
Consider a list of weights $\bfalpha=\{\alpha_1,\ldots,\alpha_n\}$, and a list $\bfu=\{u_1,\ldots,u_n\}$ of complex numbers which are roots of unity (that is, some positive power of $u_k$ is $1$). We will call the pairs $\ul{\alpha}_k=(\alpha_k,u_k)$ \emph{augmented weights}, and denote the corresponding list of 
augmented pairs by $\ubfalpha=(\bfalpha,\bfu)$.

\begin{definition}
The \emph{Verlinde sum} associated to the list of augmented weights $\ubfalpha$ is the function $V_{\ubfalpha}\colon \Lambda^\ast \times \bN \rightarrow \bC$ defined by
\begin{equation}\label{eq:verlindesum}
 V_{\ubfalpha}(\lambda,\ell)=
\sideset{}{'}\sum_{t\in \lieT_\ell} \frac{t^{\lambda}}{\prod_{(\alpha,u)\in \ubfalpha}\, 1-u\,t^{-\alpha}}.
\end{equation}

Here the  prime next to the summation symbol means that the summation only extends over elements 
$t\in \lieT_\ell$ for which  the denominator does not vanish. If all $u_k$ are equal to $1$, we also use the notation $V_{\bfalpha}$. 
\end{definition}
The functions $V_{\ubfalpha}$ are called \emph{rational trigonometric sums} in \cite{sz:res}; the term \emph{Verlinde sum} was used in  \cite{ver:sl}. Since $t^{\lambda}$ for $t\in \lieT _\ell$ depends only on the equivalence class 
of $\lambda$ modulo $\ell \Xi^\ast$, the Verlinde sum has the periodicity property 
\[ V_{\ubfalpha}(\lambda+\ell\mu,\ell)=V_{\ubfalpha}(\lambda,\ell),\ \ \ \mu\in \Xi^\ast.\]
Two special cases are worth pointing out: \begin{enumerate}
\item[(i)]If the list $\ubfalpha$ includes the \emph{trivial augmented weight} $\ul\alpha=(0,1)$,
then the summation is over an empty set, hence $V_{\ubfalpha}=0$.
\item[(ii)]
If $\ubfalpha=\emptyset$, then the denominator 
in Equation \eqref{eq:verlindesum} is equal to $1$, and the summation is over all of $\lieT_\ell$. 
We hence obtain, by finite Fourier transform, 
\begin{equation}\label{eq:verlindeseriesempty}
 V_{\emptyset}(\lambda,\ell)= (\# \lieT_\ell)\ \delta_{\ell \Xi^\ast}(\lambda)
\end{equation}
where $\delta_{\ell \Xi^\ast}(\lambda)$ equals $1$ if $\lambda\in \ell\Xi^\ast$, $0$ otherwise. 
\end{enumerate}

The Verlinde sum \eqref{eq:verlindesum} is the discrete counterpart to the \emph{multiple Bernoulli series} 
\begin{equation}\label{eq;bernoulli}
 B_{\bfalpha}(\lambda)=\sideset{}{'}\sum_{\xi\in\Xi} \frac{e^{2\pi \i \langle\lambda,\xi\rangle }}{\prod_{\alpha\in\bfalpha} 2\pi \i \langle \alpha,\xi\rangle},\ \ \ \ 
 \lambda\in\liet^\ast
 \end{equation}
where the infinite sum is defined as a generalized function of $\lambda$. These series 
satisfy $B_{\bfalpha}(\lambda+\mu)=B_{\bfalpha}(\lambda)$ for 
$\mu\in \Xi^\ast$, and turn out to be piecewise polynomial. The multiple Bernoulli series have been studied in 
\cite{bal:mul,boy:mul,sz:it,ver:sl}.

\begin{example}(Cf.~ \cite{ver:sl}\label{ex:basicexamples})
Let $\liet=\bR$, with $\Lambda=\Xi=\bZ$. Let $\bfalpha=\{1,\ldots,1\}$ (where $1$ denotes the weight $1\in \bZ=\Lambda^\ast$) be the list consisting of the element $1$ repeated $n$ times. For $n>0$, the corresponding 
Bernoulli series $B_n:=B_{\bfalpha}$ and Verlinde sum $V_n:=V_{\bfalpha}$ 
are given by 
\[ B_n(\lambda)=\sum_{j\in \bZ_{\neq 0}} \frac{e^{2\pi \i  j\lambda}}{(2\pi \i j)^n},\ \ \ \ \ \ 
V_n(\lambda,\ell)=\sum_{j=1}^{l-1} \frac{e^{2\pi \i  j{\lambda}/{\ell}}}{(1-e^{-2\pi \i  j /\ell})^n}.
\]
For $B_n$, 
this is a (generalized) function of $\lambda\in\bR$; for $V_n$, we take $\lambda$ to be 
an integer. If $n=0$ (corresponding to $\bfalpha=\emptyset$) we have 
\[ B_0(\lambda)=\delta_\bZ(\lambda),\ \ \ 
 V_0(\lambda,\ell)=\ell\delta_{\ell \bZ}(\lambda).\]
In the first formula, $\delta_\bZ$ denotes the generalized function on $\liet$ whose integration against a test function $f$ gives $\sum_{j\in \bZ}f(j)$; in the second formula, $\delta_{\ell\bZ}$ denotes the characteristic function of $\ell\bZ\subset \bZ$. The series $B_n$ is periodic with period $1$ in $\lambda$, while $V_n$ is periodic with period $\ell$. 
\label{ex:basicexamples2} The Bernoulli series satisfy a differential equation 
$\frac{\partial}{\partial \lambda}B_n(\lambda)=B_{n-1}(\lambda)-\delta_{n,1}$ and a normalization $\int_0^1 B_n(\lambda) d\lambda =0$ for $n\ge 1$. These two properties can can be used to compute $B_n$ recursively, see \cite{ver:sl}. It turns out  that $B_n(\lambda)$
are given by polynomials $\mathrm{Ber}_n(\lambda)$ on the interval $0<\lambda<1$.
We have that $\on{Ber}_0=0$, while for $n\ge 1$, $\on{Ber}_n$ is $\f{-1}{n!}$ times the  \emph{Bernoulli polynomials}. For example, 
\begin{eqnarray*}
\mathrm{Ber}_1(\lambda)&=&\frac{1}{2}-\lambda,\\ \mathrm{Ber}_2(\lambda)&=&-\frac{1}{2}\lambda^2+\frac{1}{2}\lambda-\frac{1}{12}. 
\end{eqnarray*}
Similarly, the Verlinde sums satisfy the difference equation 
	\begin{equation}\label{eq:basicrec}
	V_n(\lambda,\ell)-V_n(\lambda-1,\ell)=V_{n-1}(\lambda,\ell)-\delta_{n,1},\ \ \ n\ge 1
	\end{equation}
	and the normalization 
	\begin{equation}\label{eq:basicrec1}
	V_n(0,\ell)+\ldots+V_n(\ell-1,\ell)=0.
	\end{equation}
	By Theorem \ref{th:quasipoly} below, there are polynomials 
	$\mathrm{Ver}_n$ in $\lambda,\ell$ such that $\Ver_n(\lambda,\ell)=
	V_n(\lambda,\ell)$ for $-n<\lambda<\ell$. 
	These can be computed recursively, using \eqref{eq:basicrec} and \eqref{eq:basicrec1}. In low degrees, 
\begin{align*}
\Ver_0(\lambda,\ell)&=0,\\
\Ver_1(\lambda,\ell)&=-\lambda+\frac{\ell-1}{2},\\\Ver_2(\lambda,\ell)&=-\frac{1}{2}\lambda^2+(\frac{\ell}{2}-1)\lambda-\frac{\ell^2}{12}+\frac{\ell}{2}-\frac{5}{12}.\end{align*}
The polynomials $\on{Ber}_n(\lambda)$ are a `classical limit' of the polynomials $\on{Ver}_n(\lambda,\ell)$, in the sense that 
\begin{equation} 
\label{eqn:ClassLim}
\lim_{\ell \to \infty} \ell^{-n} \on{Ver}_n(\ell \lambda,\ell)=\on{Ber}_n(\lambda).
\end{equation}
See Section \ref{subsec:BernoulliAndVerlinde} for an explanation of this fact.

\end{example}


\subsection{Basic properties of Verlinde sums.}\label{subsec:basic}
Throughout this section, $\Lambda\subset \Xi\subset \liet$ are given lattices of maximal rank, and $\ubfalpha=(\bfalpha,\bfu)$ is a  list of augmented weights. 
We assume that $\alpha\neq 0$ for all $\alpha\in\bfalpha$. 
We will describe some general properties of the resulting Verlinde sum. 
\begin{enumerate}
	\item  \label{it:dualtorus}
	Since the function $V_{\ubfalpha}(\cdot,\ell)$ on $\Lambda^\ast$ is  $\ell \Xi^\ast$-periodic, it descends 
	to a function on the finite group $\Lambda^\ast/\ell \Xi^\ast$ (the dual group to $\lieT_\ell=l^{-1}\Xi/\Lambda$).  By definition,  this function is the finite Fourier transform of the function 
		$\lieT_\ell \rightarrow \bC$, given by 
	$t\mapsto \prod_{(\alpha,u)\in \ubfalpha}\,(1 - u\,t^{-\alpha})^{-1}$ 
	if all $u\,t^{-\alpha}\neq 1$, and $t\mapsto 0$ otherwise. 
	By inverse Fourier transform, the value of this function at $t=e$ is the sum of 
	$V_{\ubfalpha}(\cdot,\ell)$ over $\Lambda^\ast/\ell \Xi^\ast$. That is, 
\[ \sum_{[\lambda]\in \Lambda^\ast/\ell \Xi^\ast} V_{\ubfalpha}(\lambda,\ell)=\begin{cases}
\prod_{u\in\bfu} (1-u)^{-1} & \mbox{ if }1\not\in\bfu,\\
0 & \mbox{ if }1\in\bfu,\\
\end{cases}\]
where the sum picks one representative $\lambda$ from each equivalence class. 
This generalizes  \eqref{eq:basicrec1}.
\item\label{it:reality} The Bernoulli series \eqref{eq:berseries} is real-valued, since complex conjugation amounts to replacing $\xi$ with $-\xi$ in the sum. Similarly, the Verlinde sum satisfies 
\[ V_{\ubfalpha}(\lambda,\ell)^*=V_{\ubfalpha^*}(\lambda,\ell),\]
where $\ubfalpha^*$ is the list of all $(\alpha,u^*)$ with 	$(\alpha,u)\in \ubfalpha$.
\item 	\label{it:support}
One knows \cite{boy:mul} that the Bernoulli series $B_{\bfalpha}$ is supported on 
the union of (top-dimensional) admissible subspaces $\Delta\in \calS(\bfalpha)$. That is, $\mathrm{supp}(B_{\bfalpha})\subset \Delta_0+\Xi^\ast$ where 
$\Delta_0=\sp_{\bR}\bfalpha$. We will see that similarly, $\mathrm{supp}(V_{\ubfalpha})$ is contained in the set of all $\lambda,\ell$ such that 
$\lambda/\ell\in \Delta_0+\Xi^\ast$. If $\Delta_0$  is a proper subspace of $\liet^\ast$, then the Verlinde sum for $\ubfalpha=(\bfalpha,u)$ is related to lower-dimensional Verlinde sums, as follows.  
Let $\liet'=\liet/\liet_{\Delta_0}$, and denote by $\Lambda', \Xi'$ be the images of $\Lambda,\Xi$ under the quotient map, so that $(\Lambda')^*=\Lambda^*\cap \Delta_0$, and similarly for $(\Xi')^*$. 
Denote by $\bfalpha'$ the list of weights $\bfalpha$, but regarded as weights for 
$\lieT'=\liet'/\Lambda'$, and put  $\ubfalpha'=(\bfalpha',\bfu)$.
\begin{proposition}\label{prop:lowerdimensional}\label{prop:support1}
The support of the Verlinde sum satisfies 
\[ \mathrm{supp}(V_{\ubfalpha})\subset 
\left\{(\lambda,\ell)\in \Lambda^\ast\times \bN\,\Big|\ \frac{1}{\ell}\lambda\in 
 \Delta_0+\Xi^\ast\right\}.
\]
That is, $V_{\ubfalpha}(\lambda,\ell)$ is zero unless there exists $\lambda'\in (\Lambda')^*$ with $\lambda-\lambda'\in 
	\ell \Xi^\ast$. 
	In the latter case, 
	\[   V_{\ubfalpha}(\lambda,\ell)=\#(T_\ell\cap \lieT_{\Delta_0})\  V_{\ubfalpha'}(\lambda',\ell).\]
\end{proposition}
\begin{proof}
	If $t_1\in \lieT_\ell$ and $t_2\in \lieT_\ell\cap \lieT_{\Delta_0}$, then 
$ (t_1t_2)^{-\alpha}=t_1^{-\alpha}\, t_2^{-\alpha}=t_1^{-\alpha}$ for all $\alpha\in\bfalpha$. 
	We can therefore  carry out the Verlinde sum in stages, by first summing over all elements in a fixed equivalence class of $\lieT_\ell$ modulo $\lieT_\ell\cap \lieT_{\Delta_0}$, followed by a sum over $\lieT_\ell/(T_\ell\cap \lieT_{\Delta_0})$.  For fixed $t_1\in \lieT_\ell$ such that $ut_1^{-\alpha}\neq 1$ for all 
	$(\alpha,u)\in\ubfalpha$, we have that 
	\[ \sum_{t_2\in \lieT_\ell\cap \lieT_{\Delta_0}} \frac{(t_1t_2)^{\lambda}}{\prod_{(\alpha,u)\in \ubfalpha}\, 1-u\,(t_1t_2)^{-\alpha}}=\frac{t_1^{\lambda}}{ \prod_{(\alpha,u)\in \ubfalpha}\,1-u\, t_1^{-\alpha}}
	\sum_{t_2\in \lieT_\ell\cap \lieT_{\Delta_0}} t_2^{\lambda}.
	\]
	The sum on the right hand side vanishes unless $t_2^{\lambda}=1$ for all $t_2\in \lieT_\ell\cap \lieT_{\Delta_0}$, in which case it is $\# \lieT_\ell\cap \lieT_{\Delta_0}$. 
	But this condition is equivalent to 
	$\pair{\lambda}{\xi}\in \ell\bZ$ for all $\xi\in \Xi\cap \liet_{\Delta_0}$, that is, 
	$\lambda\in 
	\Lambda^\ast\cap(
	{\Delta_0}
	 + \ell \Xi^\ast)=(\Lambda')^\ast+\ell\Xi^\ast$. Assuming that this is the case, pick $\lambda'\in (\Lambda')^\ast$ such that $\lambda-\lambda'\in \ell\Xi^\ast$. 
	Then $t_1^{\lambda}=(t_1')^{\lambda'}$,  where $t_1'\in \lieT_\ell/(T_\ell\cap \lieT_{\Delta_0})$ is the image of $t_1$. Similarly, $t_1^{-\alpha_k}=(t_1')^{-\alpha'_k}$. Carrying out the sum over  
	$\lieT_\ell/(T_\ell\cap \lieT_{\Delta_0})$, we obtain $\#(T_\ell\cap \lieT_{\Delta_0})\  V_{\ubfalpha'}(\lambda',\ell)$ as desired. 
\end{proof}

\item 	\label{it:primitive}
We may always reduce to the case that the lattice vectors from the list $\bfalpha$  are \emph{primitive}, i.e, not positive multiples of shorter lattice vectors. Indeed, suppose 
$(\alpha,u)\in \ubfalpha$ with $\frac{1}{m}\alpha\in \Lambda^*$ for some integer $m>1$. Let $\tilde{\ubfalpha}$ be obtained from $\ubfalpha$ by replacing $(\alpha,u)$ 
with $(\frac{\alpha}{m},\zeta_1),\ldots(\frac{\alpha}{m},\zeta_m)$, where $\zeta_1,\ldots,\zeta_m$ are the distinct solutions of $\zeta^m=u$.  
Then  $V_{\ubfalpha}=V_{\tilde{\ubfalpha}}$, since  
\[ (1-u\,t^{-\alpha})=
\prod_{k=1}^m (1-\zeta_k\ t^{\frac{-\alpha}{m}}).
\]

\item \label{it:quasipol}
A function
$f\colon \Gamma \rightarrow \bC$
on a lattice $\Gamma$ 
is  \emph{quasi-polynomial} if there is a sublattice $\Gamma^\prime \subset \Gamma$ of finite index such that $f$ is given by a polynomial on each coset of $\Gamma^\prime$.  More generally, given a subset $S\subset \Gamma$, 
a function $f\colon S\to \bC$ is \emph{quasi-polynomial on $S$} if it is the restriction of a 
quasi-polynomial on $\Gamma$. (Note that this condition is vacuous if $S$ is a finite subset.) If $S$ contains `sufficiently many points', for example if it contains 
the intersection of $\Gamma$ with an open cone in the underlying vector space $\Gamma\otimes_\bZ \bR$, then the quasi-polynomial on $S$ extends uniquely to a quasi-polynomial on all of $\Gamma$. 

%
\begin{theorem}[Szenes\cite{sz:ver}]\label{th:quasipoly}
Suppose $\Delta$ is a top-dimensional affine subspace in $\calS(\bfalpha)$, and 
$\frakc\subset\Delta$ is a chamber (cf.~ Definition \ref{def:arrangements}).  Then the restriction of the Verlinde sum $V_{\ubfalpha}$ to the set
\begin{equation}\label{eq:region}
 \{(\lambda,\ell)\in \Lambda^\ast\times \bN|\ \lambda\in \ell \frakc - \Box \bfalpha\}\end{equation}
is quasi-polynomial. 
\end{theorem}
Szenes' paper does not state the result in this particular form (which we took from Vergne's lectures \cite{ver:sl}), but it is a simple consequence of the residue formula proved in \cite{sz:ver}. See Section \ref{subsec:szenes} below for details. Notice that 
the theorem gives a  quasi-polynomial behaviour not only on the lattice points of the cone $\bR_{>0}(\frakc\times \{1\})$, but even on a slightly larger region  
\[
 \{(x,t)\in \liet^\ast\times \bR_{>0}|\ x\in t\frakc - \Box \bfalpha\}.
\]
If $\Delta_0=\sp_{\mathbb{R}}\bfalpha=\liet^\ast$, then these regions give an open cover of $\liet^\ast \times \bR_{>0}$, as $\frakc$ varies over all open chambers. 
The appearance of the zonotope may be understood as follows: Note that if  if $\tilde{\ubfalpha}$ is obtained from $\ubfalpha$ by replacing some $(\beta,v)$ with $(-\beta,v^{-1})$, 
then 
$V_{\ubfalpha}(\lambda,\ell)=-v^{-1}\ V_{\tilde{\ubfalpha}}(\lambda+\beta,\ \ell)$. 
Hence, the quasi-polynomial behavior of $V_{\tilde\ubfalpha}$ on the cone over $\mathfrak{c}\times\{1\}$ implies the quasi-polynomial behavior of $V_{\ubfalpha}$ on the $\beta$-shifted cone. Applying this observation to all the weights, one obtains a
quasi-polynomial behavior on the shifted cone. 
\end{enumerate}


\subsection{Difference equation.}
Recall that the Verlinde sums $V_n$ in Example \ref{ex:basicexamples} satisfy a difference equation \eqref{eq:basicrec}, expressing $V_n(\lambda,\ell)-V_n(\lambda-1,\ell)$ in terms of $V_{n-1}(\lambda,\ell)$. 
We are interested in a version of this equation for the general Verlinde sums associated to a list $\ubfalpha$ of augmented weights. For $\ul\beta=(\beta,v)\in\ubfalpha$, 
denote by $\ubfalpha\backslash\ul\beta$ the list obtained by removing $\ul\beta$. 
Define a finite difference operator on functions  $f\in \Map(\Lambda^*,\bC)$ (see Appendix \ref{appsubsec:findif}):
\[ (\nabla_{\ul\beta} f)(\lambda)=f(\lambda)-v\,f(\lambda-\beta).\]
We will see that $\nabla_{\ul\beta} V_{\ubfalpha}$ is equal to $V_{\ubfalpha\backslash\ul\beta}$ modulo correction terms involving a lower-dimensional Verlinde sum. 
 Let $p\in \bN$ be the smallest natural number such that there exists $t_0\in \lieT_p$ with 
\begin{equation}\label{eq:t0}
 v\, t_0^{-\beta}=1.
 \end{equation}
For any $t\in \lieT $, denote by $\mathsf{e}_t\colon \Lambda^*\to \bC$ the map $\lambda\mapsto t^\lambda$. 
Recall the notation $\liet_\beta=\ker\beta,\ \lieT_\beta=\exp(\liet_\beta)$ from Definition \ref{def:arrangements}, and put $\Lambda_\beta=\Lambda\cap \liet_\beta$ and $\Xi_\beta=\Xi\cap \liet_\beta$.

\begin{proposition}[Difference equation]
\label{prop:differenceequation}
Given $\ul\beta\in \ubfalpha$, define $p\in \bN$ and $t_0\in \lieT_p$ by 
\eqref{eq:t0}. Then
\begin{equation}
\label{eq:deletionformula}
\nabla_{\ul\beta} V_{\ubfalpha}=
V_{\ubfalpha \sm \ul\beta}-\mathsf{e}_{t_0}\ 
 \delta_{p\bN}\ \pi^* V_{\ubfalpha'}.
\end{equation}
Here $\pi\colon \Lambda^*\to \Lambda_\beta^*$ is the 
projection, and $\ubfalpha'$ is the list of all augmented weights $(\pi(\alpha),\,t_0^{-\alpha}\,u)$ with $(\alpha,u)\in \ubfalpha\backslash \ul\beta$. 
\end{proposition}

\begin{proof}
 Under finite Fourier transform, the difference operator $\nabla_{\ul\beta}$ amounts to multiplication by $1-v t^{-\beta}$. (See Section \ref{appsubsec:finitefouriertransform}.)
 Thus, the summands in the formula for $\nabla_{\ul\beta} V_{\ubfalpha}(\lambda,\ell)$
  are the same as those for $V_{\ubfalpha \sm \ul\beta}(\lambda,\ell)$, but one is summing over $t\in \lieT_\ell$ such that $ut^{-\alpha}\neq 1$ for all $(\alpha,u)\in\bf\alpha$, whereas in the sum defining $V_{\ubfalpha \sm \ul\beta}(\lambda,\ell)$ this is only required for $(\alpha,u)\in\ubfalpha\backslash\ul\beta$. Thus, 
\begin{equation}\label{eq:fromthedef}
 V_{\ubfalpha \sm \ul\beta}(\lambda,\ell)-\nabla_{\ul\beta} V_{\ubfalpha}(\lambda,\ell)={\sum}''_{t\in \lieT_\ell}\frac{t^{\lambda}}{\prod_{(\alpha,u)\in \ubfalpha\backslash\ul\beta}\,(1-u \,t^{-\alpha})}\end{equation}
where the sum $\sum''$ is  over all $t\in \lieT_\ell$ such that 
$u t^{-\alpha}\neq 1$ for $\ul\alpha\neq \ul\beta$, but $v t^{-\beta}=1$. 
After fixing $t_0$ satisfying \eqref{eq:t0}, one can find $t\in \lieT_\ell$ with $v\,t^{-\beta}=1$ if and only if $\ell$ is a multiple of $p$, and any such $t$ 
differs from $t_0$ by an element of $\lieT_\ell\cap \lieT_\beta$. Hence, if 
 $\ell$ is not a multiple of $p$, then \eqref{eq:fromthedef}
 is an empty sum. 
If $\ell$ is a multiple of $p$, write  $t=t_0 t'$ as above. As $t$ ranges over all elements of $\lieT_\ell$ such that 
$v t^{-\beta}=1$, the elements $t'$ range over $\lieT_\ell\cap \lieT_\beta$. 
Furthermore, 
\[ t^{\lambda}=t_0^{\lambda}\, (t')^{\pi(\lambda)},\ \ \ 
u\, t^{-\alpha} =u'\, (t')^{-\alpha'}\]
for $(\alpha,u)\in\ubfalpha\backslash\ul\beta$. 
Hence the sum on the right hand side becomes 
$t_0^{\lambda}\ V_{\ubfalpha'}(\pi(\lambda),\ell)$. 
\end{proof}

\begin{remark}\begin{enumerate}
\item In Proposition \ref{prop:differenceequation}, the Verlinde sum $V_{\ubfalpha'}$ depends on the choice of $t_0$, but its product with $\mathsf{e}_{t_0}$ does not. 
\item If there exists $(\alpha,u)\neq (\beta,v)$ in the list $\ubfalpha$ such that $u t^{-\alpha}=1$ whenever 
$v t^{-\beta}=1$, then the summation in \eqref{eq:fromthedef} is over an empty set, hence the sum is zero. In particular , this is the case if the list $\ubfalpha$ contains a second copy of $\ul\beta=(\beta,v)$. 
\end{enumerate}

\end{remark}
\begin{remark}
The counterpart to \eqref{eq:deletionformula} for multiple Bernoulli series expresses the 
directional derivative $\partial_\beta B_{\bfalpha}$ in terms of 
$B_{\bfalpha\backslash\beta}$, and possibly a correction term involving a lower-dimensional Bernoulli series. See 
 \cite[Proposition 3.1]{boy:mul}.
\end{remark}

\section{Partition functions.} \label{sec:partition}
Let $\ubfalpha=(\bfalpha,\bfu)$ be a list of augmented weights. We will assume that 
all $\alpha\in \bfalpha$ are nonzero.  
Replacing summation over $\lieT_\ell$ in the definition of the Verlinde sum by an integration over $\lieT$, we obtain at a formal expression 
\[ \int_T \frac{t^{\lambda}}{\prod_{(\alpha,u)\in \ubfalpha}\,(1-u\, t^{-\alpha})}\ d t.
\]
As it stands, the integral is not well-defined
due to the zeroes of the denominator. To `regularize' the integral, choose a \emph{polarizing vector}  $\tau \in \liet$ such that $\pair{\alpha}{\tau} \ne 0$ for all $\alpha\in\bfalpha$. Then 
\begin{equation}\label{eq:weaklimit}
\lim_{\epsilon\to 0^+} 
\prod_{(\alpha,u)\in \ubfalpha}\, \big(1-u\, t^{-\alpha}\ e^{-\epsilon \pair{\alpha}{\tau}})\big)^{-1}\end{equation}
is a well-defined generalized function of $t\in \lieT $. Taking its Fourier transform, we arrive at the following definition. 
\begin{definition} The \emph{generalized partition function}
$P_{(\ubfalpha,\tau)}\colon \Lambda^\ast \rightarrow \bC$
is defined by 
\[ P_{(\ubfalpha,\tau)}(\lambda)=\lim_{\epsilon\to 0^+}\int_T \frac{t^{\lambda}}{
\prod_{(\alpha,u)\in \ubfalpha}\, (1-u \,t^{-\alpha}\ e^{-\epsilon \pair{\alpha}{\tau}})}\ d t.\]
If  all  $u=1$ for all $u\in\bfu$, we will also use the notation $ P_{(\bfalpha,\tau)}$. 
\end{definition}
Generalized partition functions were studied, e.g.,  in \cite{gu:on,gu:he,sz:res1}.
They are a discrete counterpart to the generalized Heaviside function 
\[ H_{(\bfalpha,\tau)}(\lambda)=
\lim_{\epsilon\to 0^+} \int_{\liet} \frac{ e^{2\pi \i \langle \lambda,\xi\rangle}}{\prod_{\alpha\in\bfalpha}  \langle\alpha,2\pi \i\xi+\epsilon\tau\rangle} \ d\xi,\]
here $d\xi$ is the measure on $\liet$ corresponding to the normalized Haar measure on $\lieT$, and the integral is defined as a generalized function of $\lambda\in\liet^\ast$.
\begin{example}\label{ex:basicexamples1}
We continue Example \ref{ex:basicexamples}, thus $\liet=\bR,\ \Lambda=\Xi=\bZ$, with $\bfalpha$ the list consisting of the weight $1$, repeated $n$ times. Let $\tau=1$, and write $H_n(\lambda)=H_{(\bfalpha,\tau)}(\lambda)$ and $P_n(\lambda)=P_{(\bfalpha,\tau)}(\lambda)$. One finds that 
\[ H_n(\lambda)=\begin{cases}\f{\lambda^{n-1}}{(n-1)!}& \lambda\ge 0,\\
0 & \lambda<0.\end{cases}\]
Similarly, $P_n(\lambda)$ is  the number of ways of writing $\lambda\in\Z$ as a sum 
$\lambda=k_1+\ldots+k_n$ where all $k_i\ge 0$.  Thus, 
\[ P_1(\lambda)=1,\ \ P_2(\lambda)=\lambda+1,\ \  P_3(\lambda)=\frac{1}{2}\lambda^2+\frac{3}{2}\lambda+1,\ldots\] 
for $\lambda\ge 0$, while $P_n(\lambda)=0$ for $\lambda<0$. 
\end{example}

 \medskip

Here are some of the basic properties of the generalized partition functions. 
\begin{enumerate}
\item \label{it:a}
Write $\ubfalpha=\{\ul\alpha_1,\ldots,\ul\alpha_n\}$ with $\ul\alpha_k=(\alpha_k,u_k)$. If the list $\bfalpha$ is \emph{polarized} in the sense that $\pair{\alpha_k}{\tau}>0$ for all $k$, then the integrand \eqref{eq:weaklimit} can be expanded into a geometric series, and the limit $\eps\to 0$ gives 
\[ \prod_{k=1}^n \big(\sum_{j_k=0}^\infty u_k^{j_k}\ t^{-j_k\alpha_k}\big),\]
where the sum is defined as a generalized function on $\lieT$.  
The integration against $t^{\lambda}$ picks out the coefficient of $t^{-\lambda}$ in this expansion, i.e., $P_{(\ubfalpha,\tau)}(\lambda)=\sum u_1^{j_1}\cdots u_n^{j_n}$ where the sum is over all solutions of $j_1 \alpha_1+\ldots+j_n\alpha_n=\lambda$, with
 $j_k\in\Z_{\ge 0}$. If all  $u_k=1$, then it is simply the number of such solutions; thus $P_{(\ubfalpha,\tau)}$ becomes a \emph{vector partition function}. 
 If $\bfalpha$ is the set $\mathfrak{R}_+$ of positive roots of a simple Lie algebra, and $\tau$ is in the positive Weyl chamber, then $P_{(\bfalpha,\tau)}$ is the \emph{Kostant partition function}. 
 
 In a similar fashion, $H_{(\bfalpha,\tau)}(\lambda)$ for a polarized list $\bfalpha$ is a constant (not depending on $\lambda$) times the volume of the $n-1$-dimensional polytope in $\R^n$ defined by $j_1 \alpha_1+\ldots+j_n\alpha_n=\lambda$, with
 $j_k\in\R_{\ge 0}$. See \cite{gu:on} for a detailed discussion.
 \item \label{it:b}
Suppose $\tilde\ubfalpha$ is obtained from $\ubfalpha$ by replacing some element $(\beta,v)\in \ubfalpha$ with $(-\beta,v^{-1})$. 
 Then 
\[
P_{(\ubfalpha,\tau)}(\lambda) =-v^{-1} P_{(\tilde\ubfalpha,\tau)}(\lambda+\beta).
\]
This follows from $(1-z)^{-1}=-z^{-1}(1-z^{-1})^{-1}$. 
\item \label{it:c}\label{it:partitionsupport}
Decompose the list $\ubfalpha$ into two sublists $\ubfalpha=\ubfalpha_+\cup \ubfalpha_-$, where $\alpha\in \ubfalpha_+$ if $\pair{\alpha}{\tau}>0$ and 
$\alpha\in \ubfalpha_-$ if $\pair{\alpha}{\tau}<0$. 
Let $\ubfalpha^{pol}$ be the 
\emph{polarized list}, consisting of $\ubfalpha_+$ together with all $(-\alpha,u^{-1})$ such that $(\alpha,u)\in \ubfalpha_-$. 
Define  
\[ \sigma=-\sum_{\alpha\in \bfalpha_-}\alpha.\]
Then $\pair{\sigma}{\tau}>0$, and item \eqref{it:b} above shows 
\[ P_{(\ubfalpha,\tau)}(\lambda)=\frac{(-1)^{\# \bfalpha_- }}{\prod_{u\in\bfu_-}  u}\ 
P_{(\ubfalpha^{pol},\tau)}(\lambda-\sigma).\]
In particular, using item \eqref{it:a} above, we see that the generalized partition 
function is supported in the closed cone spanned by the polarized weights, shifted by $\sigma$:
\[
\mathrm{supp}(P_{(\ubfalpha,\tau)})\subset 
\sigma+\bZ_{\ge 0}\ \bfalpha^{pol}.
 \]
\item 
In the case $\bfalpha$ spans $\t^\ast$, the functions $H_{(\bfalpha,\tau)}$ are piecewise polynomial $L^1$-functions on $\liet^\ast$, where the polynomials are supported on cones. Similarly, the partition functions $P_{(\ubfalpha,\tau)}$ exhibit a piecewise quasi-polynomial behavior. In the case of vector partition functions, this goes back to Dahmen-Micchelli 
\cite{dah:num}; see Szenes-Vergne \cite[Theorem 3.6]{sz:res1}  for the more general case. 
\item 
If $\ubfalpha$ is a disjoint union of two sublists $\ubfalpha',\ubfalpha''$, then 
\[ P_{(\ubfalpha,\tau)}=P_{(\ubfalpha',\tau)}\star P_{(\ubfalpha'',\tau)}.\]
Here the convolution (of functions on $\Lambda^\ast$) is well-defined, due to the support properties.
\item\label{it:FiniteDifferencePartition}
For $\ul\beta=(\beta,v)\in \ubfalpha$, since 
the finite-difference operator $\nabla_{\ul\beta}$ has the effect of multiplying the inverse Fourier transform by $(1-v t^{-\beta})$, we see that 
\[
\nabla_{\ul\beta} P_{(\ubfalpha,\tau)}=P_{(\ubfalpha\backslash\ul\beta,\tau)}.
\]
\item \label{it:shift}
For $h\in \lieT $, multiplication 
by the function $\mathsf{e}_h\colon \Lambda^*\to \bC,\ \lambda\mapsto h^\lambda$ has the effect of shifting the scalars of the augmented weights: That is, 
 \[
 \mathsf{e}_h \ 
 P_{(\ubfalpha,\tau)}=P_{(\tilde\ubfalpha,\tau)}\]
where $\ul{\tilde{\bfalpha}}$ consists of all $(\alpha,\,u\,h^{\alpha})$ with 
$(\alpha,u)\in\ul\bfalpha$.  
\item 
Suppose $\mathfrak{h}\subset \liet$ is a rational hyperplane,  and let $\pi\colon \Lambda^\ast\to (\Lambda\cap \mathfrak{h})^\ast=\Lambda^*/(\Lambda^*\cap \mathrm{ann}(\mathfrak{h}))$ be the quotient map.
Suppose the restrictions $\alpha|_\mathfrak{h}=\pi(\alpha)$ for $\alpha\in\bfalpha$ are all non-zero. If $\tau\in\lieh$ is polarizing for $\bfalpha$, then $\tau$ is also polarizing for the list $\bfalpha|_\mathfrak{h}$ consisting of the restrictions, and 

\[ P_{(\ubfalpha|_\mathfrak{h},\tau)}=\pi_* P_{(\ubfalpha,\tau)}\]
where  $\ubfalpha|_\mathfrak{h}=(\bfalpha|_\mathfrak{h},\bfu)$.
Here the push-forward (cf. Appendix \ref{appsubsec:pushpull}) is well-defined, since $\pi$ is proper on the support of 
$ P_{(\ubfalpha,\tau)}$ by item \eqref{it:c} above. 
\end{enumerate}

\section{Decomposition formula.}\label{sec:decomp}
Consider a list of weights $\bfalpha$, defining a Bernoulli series $B_{\bfalpha}$.  For any admissible affine subspace $\Delta\in\calS(\bfalpha)$, 
let $\bfalpha_\Delta$ be the sublist consisting of all $\alpha\in\bfalpha\cap \mathrm{ann}(\liet_\Delta)$; that is, $\alpha$ is parallel to $\Delta$. Its complement is denoted $\bfalpha^{\mathsf{c}}_\Delta=\bfalpha\backslash \bfalpha_\Delta$.

Fix an inner product on $\liet$, used to identify $\liet \simeq \liet^\ast$, and let $\gamma\in \liet$. Given $\Delta\in\calS(\bfalpha)$, let 
\[ \gamma_\Delta\in \Delta\] 
denote the orthogonal projection of $\gamma$ onto $\Delta$, and put 
\[ \tau_\Delta=\gamma_\Delta - \gamma\in\t_\Delta.\]
For a generic choice  of $\gamma$, all the projections $\gamma_\Delta$ for $\Delta\in\calS(\bfalpha)$ are in open chambers for $\Delta$, while  the $\tau_\Delta$ are polarizing vectors for the complement $\bfalpha^{\mathsf{c}}_\Delta$.

 Suppose $\mu\in \Delta$
is regular in $\Delta$, i.e., contained in one of the chambers $\frakc$ of $\Delta$. Boysal-Vergne defined 
\[ \on{Ber}_{(\bfalpha_\Delta;\mu)}\colon \liet^*\to \bC\]
to be the generalized function on $\liet^*$ with support on the affine subspace $\Delta$, which is given by a polynomial on $\Delta$ times a $\delta$-distribution in directions transverse to $\Delta$, and which coincides with $B_{\bfalpha_\Delta}$ on the chamber $\frakc$.  We will call this the \emph{polynomial germ of 
$B_{\bfalpha}$ with respect to $\frakc$}. In \cite[Theorem 9.3]{boy:mul}, A. Boysal and M. Vergne proved the following decomposition 
formula for Bernoulli series:
\begin{equation}\label{eq:boysalvergne}
 B_{\bfalpha}=\sum_{\Delta\in \calS} \on{Ber}_{(\bfalpha_\Delta;\gamma_\Delta)}\star 
H_{(\bfalpha_\Delta^{\mathsf{c}},\tau_\Delta)},
\end{equation}
(equality of generalized functions on $\liet^\ast$). The formula expresses 
the periodic generalized function $B_{\bfalpha}$ as a locally finite sum of generalized 
functions. The term indexed by $\Delta$ is supported on the affine cone $\Delta+\R_{\ge 0}\{\bfalpha_\Delta^{\mathsf{c}}\}$; a given point of $\liet^*$ is contained in only finitely many of these affine cones. In particular, 
none of the cones with $\Delta\neq\liet^*$ contains the center of expansion. If $\Delta_0=\sp(\bfalpha)=\liet^\ast$, then $\Delta_0$ gives a leading polynomial contribution supported on $\liet^\ast$, while the other terms give successive corrections away from the center of expansion. 

In this section, we will give an analogous result for the Verlinde sums. Let $\ubfalpha$ be a list of augmented weights.
For $\Delta\in\calS(\bfalpha)$,  denote by  $\ubfalpha_\Delta=(\bfalpha_\Delta,\bfu_\Delta)$ the sublist consisting of all $(\alpha,u)\in\ubfalpha$ such that $\alpha\in \mathrm{ann}(\liet_\Delta)$. 
According to Theorem \ref{th:quasipoly}, if $\frakc$ is a chamber of $\Delta$, the Verlinde sum $V_{\ubfalpha_\Delta}$ is 
quasi-polynomial on the set of $(\lambda,\ell)$ such that $\lambda\in\ell\frakc$. 
Given $\mu\in \frakc$, we define 
\begin{equation}\label{eq:qpergerm}
\Ver_{(\ubfalpha_\Delta;\mu)}\colon \Lambda^\ast\times\bN\to \bC \end{equation}
 to be the unique function supported on the set of $(\lambda,\ell)$ with $\lambda\in \ell\Delta$, such that this function is  quasi-polynomial on this set and agrees with $V_{\ubfalpha_\Delta}$ on the set of all $(\lambda,\ell)$ with $\lambda\in \ell\mathfrak{c}$. We refer to $\Ver_{(\ubfalpha_\Delta;\mu)}$ as the  \emph{quasi-polynomial germ} of $V_{\ubfalpha_\Delta}$ at $\mu$. 
 
\begin{remark}\label{rem:larger}
By Theorem \ref{th:quasipoly}, the Verlinde sum $V_{\ubfalpha_\Delta}$ agrees with $\Ver_{(\ubfalpha_\Delta;\mu)}$ on the slightly larger set where
$\lambda\in \ell\mathfrak{c}-\Box \bfalpha_\Delta$. 
\end{remark}


Pick a generic $\gamma\in \liet$, and define $\gamma_\Delta,\ \tau_\Delta$ for $\Delta\in \calS(\bfalpha)$ as above. The polarizing vectors $\tau_\Delta$ define partition functions $P_{(\ubfalpha_\Delta^{\mathsf{c}},\tau_\Delta)}$, with support in a 
shifted cone $\sigma_\Delta+\bR_{\ge 0}(\bfalpha_\Delta^{\mathsf{c}})^{pol}$ where 
(see Section \ref{sec:partition}\eqref{it:partitionsupport})
\begin{equation}\label{eq:sigmadelta}
\sigma_\Delta=-\sum_{\alpha\in (\bfalpha^{\mathsf{c}}_{\Delta})_-}\alpha,\end{equation}
a sum over weights $\alpha$ in $\bfalpha_\Delta^{\mathsf{c}}$ such that $\pair{\alpha}{\tau_\Delta}<0$. Then $\l\sigma_\Delta,\tau_\Delta\r \ge 0$; in fact 
$\langle\mu,\tau_\Delta\rangle\ge 0$ for all elements in this shifted cone. 
\begin{theorem}[Decomposition formula for Verlinde sums]
\label{DecompositionTheorem}
\label{th:decomposition}
Let $\ubfalpha$ be a list of augmented weights, and let $\gamma \in \liet^\ast$ be generic.  Then
\begin{equation}\label{eq:decompositionformula}
V_{\ubfalpha}=\sum_{\Delta \in \calS}\Ver_{(\ubfalpha_\Delta;\gamma_\Delta)}\star P_{(\ubfalpha_\Delta^{\mathsf{c}},\tau_\Delta)}
\end{equation}
(a locally finite sum). Here the term   indexed by $\Delta$ is 
is a (well-defined) convolution of functions on $\Lambda^\ast$, for fixed $\ell\in\bN$.
\end{theorem}
\begin{proof}
We first show that the right hand side of the decomposition formula is locally finite, and hence well defined. 
The first factor in each term 
\begin{equation}\label{eq:term}
 V_{\ubfalpha}^\Delta=\Ver_{(\ubfalpha_\Delta;\gamma_\Delta)}\star P_{(\ubfalpha_\Delta^{\mathsf{c}},\tau_\Delta)}
\end{equation}
is supported in the set of $(\lambda,\ell)$ such that $\lambda\in \ell\Delta$, while according to Section \ref{sec:partition}
\eqref{it:partitionsupport}, the second factor is supported in $\sigma_\Delta+\bR_{\ge 0}(\bfalpha_\Delta^{\mathsf{c}})^{pol}$, 
using the polarization of $\bfalpha_\Delta^{\mathsf{c}}$ defined by 
$\tau_\Delta$. This shows that the convolution has support in the 
the set of all $(\lambda,\ell)\in \Lambda^\ast\times \bN$ such that $\lambda$ is contained in the affine cone 
\begin{equation}\label{eq:supportconvol}
\sigma_\Delta+\ell\Delta+\bR_{\ge 0}(\bfalpha_\Delta^{\mathsf{c}})^{pol}.
\end{equation}
For fixed $\lambda\in \Lambda^*$, there are only finitely many  $\Delta$ such that 
$\lambda$ is contained in \eqref{eq:supportconvol}. Hence the sum in \eqref{eq:decompositionformula} is locally finite.

The proof of the equality $V_{\ubfalpha}=\sum_{\Delta} V_{\ubfalpha}^\Delta$ 
involves an induction on the number $n$ of elements in $\bfalpha$. 
Consider first the base case $n=0$, so that the admissible subspaces are all $0$-dimensional: $\calS(\emptyset)=\Xi^\ast$.  For $\Delta=\{\mu \}$ with $\mu \in \Xi^\ast$,  we have that $\lieT_\Delta=\lieT$, $\gamma_\Delta=\mu$, and consequently 
$V_{\emptyset}^\Delta(\lambda,\ell)=\# \lieT_\ell\ \delta_{\ell\mu}(\lambda)$. 
Summation over all $\Delta=\{\mu\}$ gives the expression \eqref{eq:verlindeseriesempty} for $V_\emptyset$.

Suppose now that $\bfalpha$ has $n>0$ elements, and that the decomposition formula has been proved for lists with at most $n-1$ elements. The argument for $V_{\bfalpha}$ splits into two stages. In the first step, we will show that the difference $ V_{\ubfalpha}-\sum_\Delta V_{\ubfalpha}^\Delta$ 
is annihilated by all difference operators $\nabla_{\ul\beta}$, for  $\ul\beta \in \ubfalpha$. By definition of the difference operator, 
this implies that for any given $\ell$, this difference is quasi-invariant under the action of the sublattice $\bZ\bfalpha\subset \Lambda^\ast$, i.e. invariant up to some character 
$\bZ\bfalpha\to \mathrm{U}(1)$. 
 In the second step, we show that $V_{\bfalpha}$ coincides with $ \sum_\Delta V_{\ubfalpha}^\Delta $ on the lattice points of some fundamental domain for the action of $\bZ\bfalpha$. Together, these two steps show that they agree everywhere.


\bigskip

\noindent{\bf Step 1.}  Claim: For all $\ul\beta\in\ubfalpha$, we have that $\nabla_{\ul\beta} V_{\ubfalpha}=\sum_{\Delta\in\calS(\bfalpha)} \nabla_{\ul\beta} V_{\ubfalpha}^\Delta$.
\bigskip

Recall the difference equation \eqref{eq:deletionformula}: 
$
\nabla_{\ul\beta} V_{\ubfalpha}=
V_{\ubfalpha \sm \ul\beta}-\mathsf{e}_{t_0}\ 
 \delta_{p\bN}\ \pi^* V_{\ubfalpha'}$. 
To compare with the sum over all $\nabla_{\ul\beta} V_{\ubfalpha}^\Delta$ with $\Delta\in\calS(\bfalpha)$, we consider two cases: 

{\bf Case (i):} $\ul\beta\in \ubfalpha_\Delta^{\mathsf{c}}$. Apply $\nabla_{\ul\beta}$ to \eqref{eq:term}, using the property $ \nabla_{\ul\beta}(f\star g)=f\star\nabla_{\ul\beta}(g)$ of the convolution. 
Since  $\nabla_{\ul\beta} P_{(\ubfalpha _\Delta^{\mathsf{c}},\tau_\Delta)}=P_{(\ubfalpha_\Delta^{\mathsf{c}} \sm \ul\beta;\,\tau_\Delta)}$ by Section \ref{sec:partition}\eqref{it:FiniteDifferencePartition}, and since
$\ubfalpha_\Delta=(\ubfalpha \sm \ul\beta)_\Delta$, we obtain  
\begin{equation} \label{eq:ja}\nabla_{\ul\beta} V_{\ubfalpha}^\Delta  
= V_{\ubfalpha\backslash\ul\beta}^\Delta.
\end{equation}

{\bf  Case (ii):} $\ul\beta\in \ubfalpha_\Delta$.  
Apply $\nabla_{\ul\beta}$ to \eqref{eq:term}, using the property $ \nabla_{\ul\beta}(f\star g)=\nabla_{\ul\beta}(f)\star g$  of the convolution.
 Lemma \ref{VerlindeFiniteDifference} below gives
\begin{equation}\label{eq:ja2}
 \nabla_{\ul\beta} \Ver_{(\ubfalpha_\Delta;\gamma_\Delta)}=
\Ver_{(\ubfalpha_\Delta\backslash\ul\beta;\gamma_\Delta)}-
\mathsf{e}_{t_0}\  \delta_{p\bN}\ \pi^*\,\Ver_{(\ubfalpha'_{\pi(\Delta)};\pi(\gamma_\Delta))}.
\end{equation}
Taking a convolution with $P_{(\ubfalpha_\Delta^{\mathsf{c}},\tau_\Delta)}$, the first term produces $V_{\ubfalpha\backslash\ul\beta}^\Delta$. For the second term, we use the property  $\mathsf{e}_t (f\star g)=(\mathsf{e}_t\ f)\star (\mathsf{e}_t\ g)$ and 
$(\pi^* f)\star g=\pi^* (f\star \pi_* g)$ 
of the convolution. Thus
\begin{equation}\label{eq:ja3}
 \nabla_{\ul\beta}  V_{\ubfalpha}^\Delta=
V_{\ubfalpha\backslash\ul\beta}^\Delta-\delta_{p\bN}\ \mathsf{e}_{t_0}\ \pi^*\big(\Ver_{(\ubfalpha'_{\pi(\Delta)};\pi(\gamma_\Delta))}\ast 
\pi_* \big(\mathsf{e}_{-t_0}P_{(\ubfalpha_\Delta^{\mathsf{c}},\tau_\Delta)}\big)
\big).
\end{equation}
Multiplying the partition function 
 $P_{(\ubfalpha_\Delta^{\mathsf{c}},\tau_\Delta)}$ 
by $\mathsf{e}_{t_0^{-1}}$ amounts to replacing every $(\alpha,u)\in \ubfalpha_\Delta^{\mathsf{c}}$
with $(\alpha,\,t_0^{-\alpha} u)$, while push-forward $\pi_\ast$ amounts to a restriction to $\liet_\beta$. 
Hence
\[ \pi_\ast\big(\mathsf{e}_{t_0^{-1}}\  P_{(\ubfalpha_\Delta^{\mathsf{c}},\tau_\Delta)}\big)
=P_{((\ubfalpha'_{\pi(\Delta)})^{\mathsf{c}};\pi(\tau_\Delta))}.
\]
The inner product on $\liet$ restricts to an inner product on $\liet_\beta\simeq \liet_\beta^\ast$ satisfying $\pi(\gamma_\Delta)=\pi(\gamma)_{\pi(\Delta)}$, 
which implies that
\[ \tau_{\pi(\Delta)}:=\pi(\tau_\Delta)=\pi(\gamma)_{\pi(\Delta)}-\pi(\gamma)\]
is the polarizing vector defined by $\pi(\gamma)$. 
We hence arrive at
\begin{equation}\label{eq:ja1}
\nabla_{\ul\beta} V_{\ubfalpha}^\Delta= V_{\ubfalpha\backslash\ul\beta}^\Delta
-\mathsf{e}_{t_0}\ \delta_{p\bN}\ \pi^\ast V_{\ubfalpha'}^{\pi(\Delta)},
\end{equation}
where the generic element used in defining $V_{\ubfalpha'}^{\pi(\Delta)}$ is $\pi(\gamma)$.

Having worked out $\nabla_{\ul\beta} V_{\ubfalpha}^\Delta$ for the two cases, 
let us now take  the sum over all admissible affine subspaces $\Delta\in\calS(\bfalpha)$.

Observe that if $\Delta\not \in \calS(\bfalpha\backslash\beta)$, then 
the term $V_{\ubfalpha\backslash\ul\beta}^\Delta$ is zero. Indeed, $\Delta\not \in \calS(\bfalpha\backslash\beta)$ means that 
$\mathrm{ann}(\liet_\Delta)$ is spanned by a subset of $\bfalpha$, but not by a subset of $\bfalpha \sm \beta$. As a consequence, the restriction of the Verlinde sum 
$V_{\bfalpha_\Delta\backslash\beta}$ to $\Delta$ is supported on a union of affine subspaces of codimension at least one in $\Delta$, and consequently $\Ver\big(\alpha'_{\pi(\Delta)};\gamma_\Delta\big)=0$, thus 
$V_{\ubfalpha\backslash\ul\beta}^\Delta=0$. 

On the other hand, 
the projection $\pi\colon \liet^\ast\to \liet_\beta^\ast$ gives a bijection from the set of $\Delta\in \calS(\bfalpha)$ for which $\beta\in \bfalpha_\Delta$ (case (ii) above), with the set $\calS(\bfalpha')$ of admissible affine subspaces for the list $\bfalpha'$. From equations \eqref{eq:ja}, \eqref{eq:ja1} we therefore obtain 
\begin{equation}\label{eq:almostthere}
\nabla_{\ul\beta}\ \sum_{\Delta\in\calS(\bfalpha)}  V_{\ubfalpha}^\Delta
=\sum_{\Delta\in\calS(\bfalpha\backslash\beta)}V_{\ubfalpha\backslash\ul\beta}^\Delta
-\mathsf{e}_{t_0}\ \delta_{p\bN}\ \pi^\ast 
\sum_{\Delta'\in \calS(\bfalpha')}
V_{\ubfalpha'}^{\Delta'}.\end{equation}
Applying the induction hypothesis to the right hand side, this is becomes 
\begin{equation}\label{eq:almostthere}
\nabla_{\ul\beta}\ \sum_{\Delta\in\calS(\bfalpha)}  V_{\ubfalpha}^\Delta
=V_{\ubfalpha\backslash\ul\beta}
-\mathsf{e}_{t_0}\ \delta_{p\bN}\ \pi^\ast 
V_{\ubfalpha'}.\end{equation}
Finally, by the difference equation \eqref{eq:deletionformula} this equals $\nabla_{\ul\beta} V_{\ubfalpha}$. \bigskip

\noindent{\bf Step 2.}  Claim:  $V_{\ubfalpha}$ agrees with the sum  $\sum_{\Delta} V_{\ubfalpha}^\Delta$ 
on a set of representatives  for the translation action of the lattice $\bZ \bfalpha\subset \Lambda^\ast$ 
in $\Lambda^\ast\times \bN$. 
\bigskip

Recall that by Proposition \eqref{prop:support1}, the Verlinde sum $V_{\ubfalpha}(\cdot,\ell)$ for fixed $\ell\in\bN$ is supported on the union of all $\ell\Delta_1$, for $\Delta_1\in \calS(\bfalpha)$ top-dimensional (i.e., a $\Xi^\ast$-translate of  $\Delta_0=\sp\,\bfalpha$). The same is true for  each of the summands $V_{\ubfalpha}^\Delta$. Since each $\ell\Delta_1$ is  invariant under the $\bZ\bfalpha$-action, 
 it suffices to show that the functions $V_{\bfalpha}(\cdot,\ell)$ and $\sum_\Delta V^\Delta_{\bfalpha}(\cdot,\ell)$ agree on a set of representatives for the $\bZ \bfalpha\subset \Lambda^\ast$ action on $\ell\Delta_1$, for any given $\ell$ and $\Delta_1$. The set 
\begin{equation}
\label{AgreementSet}
 \Lambda^*\cap \big(\ell\, \gamma_{\Delta_1} - \Box \bfalpha\big)
\end{equation}
contains such a set of representatives. By Remark \ref{rem:larger}, and the definition of the quasi-polynomial germ,  $V_{\ubfalpha}(\cdot,\ell)$ and  $\Ver_{(\ubfalpha_{\Delta_1};\gamma_{\Delta_1})}(\cdot,\ell)$ agree on \eqref{AgreementSet}. 
Since $\bfalpha_{\Delta_1}^{\mathsf{c}}=\emptyset$, the corresponding partition function $P_{(\ubfalpha_{\Delta_1}^{\mathsf{c}},\tau_{\Delta_1})}$ is simply $\delta_0\colon \Lambda^*\to \bC$. 
That is, 
\[ V_{\ubfalpha}(\cdot,\ell)=V_{\ubfalpha}^{\Delta_1}(\cdot,\ell)
\]
on \eqref{AgreementSet}. We claim that the supports of all other terms 
$V_{\ubfalpha}^\Delta(\cdot,\ell)$ with $\Delta\neq \Delta_1$ do not meet \eqref{AgreementSet}. This is immediate when 
$\Delta$ is not contained in $\Delta_1$, thus suppose $\Delta\subset \Delta_1$. 
Note that 
\begin{equation}\label{eq:itsnormal} \mathsf{n}:=\gamma_\Delta-\gamma_{\Delta_1}
=\tau_\Delta-\tau_{\Delta_1}
\end{equation}
is a normal vector to $\Delta$ as an affine subspace of $\Delta_1$. 
Since $\tau_{\Delta_1}$ is normal to $\Delta_1$, it is orthogonal to all 
$\alpha\in \bfalpha$. In the defining the polarization of $\bfalpha_\Delta^{\mathsf{c}}$, and hence in the description \eqref{eq:supportconvol}
of $\mathrm{supp}(V_{\ubfalpha}^\Delta(\cdot,\ell))$, we may therefore replace $\tau_\Delta$ with $\mathsf{n}$. It follows that the elements $\lambda\in\Lambda^\ast$ in the support of $V_{\ubfalpha}^\Delta(\cdot,\ell)$ satisfy  
\[\pair{\lambda-\ell\gamma_{\Delta_1}}{\mathsf{n}}>
\pair{\lambda-\ell\gamma_\Delta}{\mathsf{n}}\ge \pair{\sigma_\Delta}{\mathsf{n}}.\]	
On the other hand, if 
$\lambda\in \ell\gamma_{\Delta_1} -\sum_{\alpha\in\bfalpha}[0,1]\alpha$, then 
\[\pair{\lambda-\ell\gamma_{\Delta_1}}{\mathsf{n}}\le 
-\sum_{\alpha\in (\bfalpha^{\mathsf{c}}_{\Delta})_-} \l\alpha,\mathsf{n}\r=
\pair{\sigma_\Delta}{\mathsf{n}}.
 \]
Hence, these elements are never in the support of  $V_{\ubfalpha}^\Delta(\cdot,\ell)$ when $\Delta$ is properly contained in $\Delta_1$.  

This concludes  the argument for step 2, and completes the proof of the theorem. 
\end{proof}

In the proof we used the following difference equation for the quasi-polynomial germs, 
which is obtained as a consequence of the difference equation for Verlinde sums in 
Proposition \ref{prop:differenceequation}. Using the notation 
from that proposition we have:  
\begin{lemma}\label{VerlindeFiniteDifference}
	 Suppose $\ul\beta=(\beta,v) \in \ubfalpha_\Delta$, and let  $p\in \bN$ be the smallest number such that there exists $t_0\in \lieT_p$ with $v\,t_0^{-\beta}=1$.  
	 Let $\ubfalpha'_{\pi(\Delta)}$ be the list of augmented weights $(\pi(\alpha),\ u\,t_0^{-\alpha})$ for $(\alpha,u)\in\bfalpha_\Delta$ with $\ul\alpha\neq \ul\beta$.  
	    Then
	 \[ \Ver_{(\ubfalpha_\Delta \sm \ul\beta;\mu)}=\nabla_{\ul\beta} \Ver_{(\ubfalpha_\Delta;\mu)}+\mathsf{e}_{t_0}\ \delta_{p\bN}\ \pi^*\Ver_{(\ubfalpha'_{\pi(\Delta)};\pi(\mu))}.\]
	 Here $\pi\colon \Lambda^*\to \Lambda_\beta^*$ is the projection. 
\end{lemma}
\begin{proof}
	On the set of all $(\lambda,\ell)$ such that $\lambda\in\ell\mathfrak{c}$, 
	this formula coincides with the difference equation for Verlinde sums (Proposition \ref{prop:differenceequation}, with $\ubfalpha$ replaced by $\ubfalpha_\Delta$). 
	Since all terms in the deletion formula are quasi-polynomial on this set, the result follows. 
\end{proof}

\begin{example}
We illustrate the decomposition formulas for the Bernoulli series 
$B_n(\lambda)$ and the Verlinde sums $V_n(\lambda,\ell)$, using the 
notation from Examples \ref{ex:basicexamples}, \ref{ex:basicexamples1}. 
Take $\gamma\in \mathfrak{c}=(0,1)$. The contribution from  $\Delta_0=\liet^\ast$ give the leading term
\[ H^{\Delta_0}_{\bfalpha}(\lambda)=\Ber_n(\lambda),\ \ \ \ 
V^{\Delta_0}_{\bfalpha}(\lambda,\ell)=\Ver_n(\lambda,\ell). \]
The other contributions come from $\Delta=\{\mu\}$ with $\mu\in \Xi^\ast=\bZ$. 
We have $\bfalpha_\Delta=\emptyset$, hence 
 \[ \Ber_{(\bfalpha_\Delta,\gamma_\Delta)}(\lambda)=\delta_\mu(\lambda),\ \ 
 \Ver_{(\bfalpha_\Delta,\gamma_\Delta)}(\lambda,\ell)=\ell \delta_{\ell\mu}(\lambda).\] 
Since $\tau_\Delta=\mu-\gamma$, the weight $\alpha=+1$ is polarized if and only if $\mu>0$. With the partition functions from Example \ref{ex:basicexamples1}, this gives 
\[ H_{(\bfalpha^{\mathsf{c}}_\Delta,\tau_\Delta)}(\lambda)=H_n(\lambda),\ \ \ 
P_{(\bfalpha^{\mathsf{c}}_\Delta,\tau_\Delta)}(\lambda)=P_n(\lambda)\]
for $\Delta=\{\mu\}$ with $\mu >0$ while 
\[ H_{(\bfalpha^{\mathsf{c}}_\Delta,\tau_\Delta)}(\lambda)=(-1)^n H_n(-\lambda),\ \ \ 
P_{(\bfalpha^{\mathsf{c}}_\Delta,\tau_\Delta)}(\lambda)=(-1)^n P_n(-\lambda-n)\]
for $\Delta=\{\mu\}$ with $\mu \le 0$.
The resulting decomposition formulas hence read as 
\[ B_n(\lambda)=\on{Ber}_n(\lambda)+\sum_{n=1}^\infty H_n(\lambda-\mu)+(-1)^n \sum_{\mu=-\infty}^0 H_n(\mu-\lambda)\]
and 
\[ V_n(\lambda,\ell)=\Ver_n(\lambda,\ell)+\ell\sum_{\mu=1}^\infty P_n(\lambda-\ell\mu)+(-1)^n \ell \sum_{\mu=-\infty}^0 P_n(\ell\mu-\lambda-n),
\]
respectively.
\end{example}
See also Example \ref{SzenesSU(3)Example} further below, where we describe the decomposition formula for Example \ref{ex:lie2} in the case $G=SU(3)$.

\section{Szenes' Residue Theorem.}\label{subsec:szenes}

In this Section, we will review Szenes' residue theorem for the Verlinde sums \cite{sz:ver} (see also \cite{sz:co}).  The formula is a powerful tool for explicit computations of such sums. We will explain how it implies the quasi-polynomial behaviour, in the form stated in Theorem \ref{th:quasipoly}. 

\subsection{The constant term functional.}
The formula in \cite{sz:ver} is expressed in terms of iterated `constant-term' functionals.  We describe these briefly here; for further background see \cite{br:ar1,sz:it,sz:ver,sz:res1}.  For a meromorphic function $f$ of a single variable $z$, 
defined on a neighborhood of $0 \in \bC$, the constant term 
$\CT_{z=0}(f)$ is defined to be the constant term in the Laurent expansion of $f(z)$ about the point $0$. In terms of residues,
\begin{equation}\label{eq:CT}
\CT_{z=0}(f)=\mathrm{Res}_{z=0}\Big(f(z)\frac{d z}{z}\Big).
\end{equation} 

Let $\Lambda\subset \liet$ be a maximal rank lattice, and 
consider a $\Lambda$-periodic affine hyperplane arrangement $\calA$ in $\liet$. 
For $p\in \liet$ we denote by $\calA_p\subset\calA$ the collection of affine hyperplanes passing through $p$. The point $p$ is called a \emph{vertex} of the arrangement if the intersection of affine hyperplanes in $\calA_p$ is $\{p\}$. Let $\vx(\calA)$ denote the set of all vertices. A subset of $\calA_p$ is called \emph{linearly independent} if the conormal directions to the hyperplanes in this subset are linearly independent.

Let $M_\calA$ be the ring of meromorphic functions on $\liet_\bC$ whose singularities 
are contained in $\bigcup_{H\in\calA} H$.  Suppose $p\in \vx(\calA)$, and  consider an ordered $r$-tuple $\bfH=(H_1,\ldots,H_r)$ of $r=\dim\liet$ linearly independent affine hyperplanes $H_k\in \calA_p$. The \emph{iterated constant term functional} 
\[ \iCT_{\bfH}\colon M_\calA\to  \bC \]
is defined as follows: 
Choose affine-linear functions $\beta_k$ with $H_k=\beta_k^{-1}(0)$; the collection 
$\bfbeta=(\beta_1,\ldots,\beta_r)$ defines an isomorphism $\bfbeta\colon \liet_\bC\to \bC^r$, taking $p$ to $0$.  
Thus $f\circ \bfbeta^{-1}$ is a function of $z_1,\ldots,z_r$. One defines 
\begin{equation}
\label{eq:iCT}\iCT_{\bfH}(f) = \CT_{z_1=0}\ \cdots \CT_{z_r=0}\, (f\circ \bfbeta^{-1}).
\end{equation}
That is, one first computes $\CT_{z_r=0}(f\circ \bfbeta^{-1})$ as a meromorphic function of $z_1,\ldots ,z_{r-1}$, one next computes $\CT_{z_{r-1}=0}$ of the result, and so forth. 
This definition is  independent of the choice of $\beta_k$ defining $H_k$, 
due to the fact that the constant term \eqref{eq:CT} is invariant under coordinate changes of the form $z\leadsto cz$. On the other hand, 
it does in general depend  on the ordering of the hyperplanes. Note that the definition of $\iCT_{\bfH}$ extends to the space $M_{\calA,p}$ of \emph{germs} at $p$ of meromorphic functions with poles in $\calA$. 

\subsection{nbc bases.}\label{subsec:nbc}
Szenes' Theorem \ref{SzenesFormulaTheorem} expresses the Verlinde sum $V_{\bfalpha}$ as a finite sum over iterated residues, for suitable 
$r$-tuples $\bfH$ of affine hyperplanes.  Given $p\in \vx(\calA)$, choose an ordering on the set $\calA_p$. This determines  a unique collection $B_p$ of $r$-tuples $\bfH$ of linearly independent hyperplanes in $\calA_p$, with the following property: For every $H\in \calA_p$, the set 
\[ \{H'\in \bfH\colon H<H'\}\cup H\]
is linearly independent. The collection $B_p$ is known 
in the theory of hyperplane arrangements as an \emph{nbc basis} of $\calA_p$, where nbc stands for \emph{no broken circuit basis} 
 (cf.~  \cite{sz:ver,sz:it}).

\begin{example}
If $\dim(\liet)=2$, the nbc basis $B_p$ for a given ordering of $\calA_p$ is the collection of all 2-tuples $\bfH=(H_0,H)$, where $H_0$ is the unique smallest element of $\calA_p$, and 
$H$ is any of the other elements. 
\end{example}

\begin{remark}
\begin{enumerate}
\item In \cite{br:ar1}, it is shown that $B_p$ corresponds naturally to a vector space basis for the subspace of \emph{simple fractions}, in the ring of rational functions generated by the inverses of the affine-linear functions defining the hyperplanes in $\calA_p$.
\item In Szenes' formula one can also use more general \emph{orthogonal bases}, which need not arise from a linear ordering of $\calA_p$, see \cite{sz:ver}.  For simplicity we have only described nbc bases, which are a special case.
\item For interesting and non-trivial examples, one might consult \cite{bal:mul} where the case of hyperplane arrangements arising from root systems of the classical Lie algebras are discussed in detail, as well as many computations of the associated multiple Bernoulli series.  
\end{enumerate}
\end{remark}

\subsection{Statement of Szenes' theorem.}
For the remainder of this section, we assume that $\Xi\subset \liet$ is a lattice containing $\Lambda$, and that $\ubfalpha=(\bfalpha,\bfu)$ is a list of augmented weights, where 
the weights $\alpha\in\bfalpha$ are all non-zero.  Let $\calA=\calA(\ubfalpha)$ be the resulting affine hyperplane arrangement in $\liet$, consisting of the affine hyperplanes $\alpha^{-1}(s)$ such that $s\in\bQ$ with $(\alpha,e^{2\pi \i s})=(\alpha,u)\in \ubfalpha$. 
Using the procedure from Subsection \ref{subsec:basic} \eqref{it:primitive}, we  will assume that the $\alpha\in\bfalpha$ are \emph{primitive} vectors in the weight lattice. Note that this procedure does not change the arrangement $\calA(\ubfalpha)$. 

 We will also assume $\sp_{\bR} \bfalpha=\liet^\ast$, so that the set $\vx(\calA)$ of vertices is non-empty. For every vertex $p\in \vx(\calA)$, choose an ordering of the set $\calA_p$, and let 
$B_p$ be the resulting nbc basis, as in Subsection \ref{subsec:nbc}.  Recall also the set $\calS=\calS(\bfalpha)$ of admissible affine subspaces $\Delta\subset \liet^\ast$, and that 
a \emph{chamber} is a component of  $\liet^\ast-\bigcup_{\Delta\neq \liet^\ast}\Delta$  (cf. Definition \ref{def:arrangements}).

For any  chamber $\frakc\subset \liet^\ast$, and any $\bfH\in B_p$ and $(\lambda,\ell)\in \Lambda^\ast\times \bN$, define a meromorphic function 
\[ X\mapsto T_{\frakc,\bfH}(\lambda,\ell)(X)\]
of $X\in \liet_\bC$, as follows. Choose affine-linear functions $\beta\colon \liet\to \bR$ defining the 
hyperplanes $H\in \bfH$, in such a way that their linear parts
$\beta^0=\beta-\beta(0)$ lie in $\Xi^\ast\subset \Lambda^\ast$. 
Let $\bfbeta$ be the list of these functions; the list of their 
linear parts forms a basis $\bfbeta^0$ of $\liet^\ast$.   Let $\vol_{\Xi^\ast}(\Box \bfbeta^0)$ be the volume of the zonotope with respect to Lebesgue measure on $\liet^\ast$ (normalized such that  $\liet^\ast/\Xi^\ast$ has volume $1$). Equivalently,
\[ \vol_{\Xi^\ast}(\Box \bfbeta^0)=\# \big(\Xi^\ast \cap (\nu+\Box \bfbeta^0)\big)\] 
for generic $\nu\in\liet^*$. (The intersection on the right hand side does not change as $\nu$ varies 
 in a chamber of $\calS(\bfbeta^0)$; in particular, we may also write it as $\Xi^\ast\cap (\frakc+\Box \bfbeta^0)$.) With these preparations, put
\begin{equation}\label{eq:tcbeta}
 T_{\frakc,\bfH}(\lambda,\ell)(X)=
\frac{1}{ \vol_{\Xi^\ast}(\Box\bfbeta^0) }\sum_\mu e^{2\pi \i \pair{\lambda-\ell\mu}{X}}
\prod_{\beta \in \bfbeta} \frac{2\pi \i \ell \beta(X)}{1-e^{2\pi \i \ell \pair{\beta^0}{X}}};
\end{equation}
here the sum is over all $\mu \in \Xi^\ast\cap (\frakc-\Box \bfbeta^0)$. As shown by Szenes \cite{sz:ver}, this function does not 
depend on the choice of affine-linear functions $\beta$ defining $H$.

Observe that the function \eqref{eq:tcbeta} is holomorphic for $X$ near $p$. 
Indeed, the singularities of $X\mapsto (1-e^{2\pi \i \ell \pair{\beta^0}{X}})^{-1}$ 
are contained in the subset where $\ell \pair{\beta^0}{X}\in\bZ$; but for $X=p$, any such singularity is compensated by the zero of $X\mapsto \beta(X)$ in the enumerator.
Let $f_{\ubfalpha}\in M_\A$ be the meromorphic function defined by 
\begin{equation}\label{eq:falpha}
f_{\ubfalpha}(X)=\prod_{(\alpha,u)\in \ubfalpha} \Big(1-u\, e^{-2\pi \i\pair{\alpha}{X}} \Big)^{-1}.
\end{equation}
We are now in position to state Szenes' theorem for Verlinde sums:
\begin{theorem}[Szenes {\cite[Theorem 4.2]{sz:ver}}]
\label{SzenesFormulaTheorem}
Let $\ubfalpha$ be a list of augmented weights, where all $\alpha\in\bfalpha$ are primitive lattice vectors in  $\Lambda^\ast$, such that 
$\sp_{\bR}\bfalpha=\liet^\ast$. Define $\calA=\calA(\ubfalpha)$ as above.   Then for every 
chamber $\mathfrak{c}$ of $\calS(\bfalpha)$, and all $(\lambda,\ell)\in \Lambda^*\times\bN$ with  $\lambda \in \ell \frakc - \Box \bfalpha$, 
the Verlinde sum  is given by the formula 
\begin{equation}
\label{SzenesFormula}
V_{\ubfalpha}(\lambda,\ell)=\sum_{[p] \in \vx(\calA)/\Lambda} \, \, \, \, \sum_{\bfH \in B_p} 
\iCT_{\bfH}\Big(T_{\frakc,\bfH}(\lambda,\ell) f_{\ubfalpha}\Big).
\end{equation}
Here the summation over $[p]$ uses one representative $p\in \vx(\calA)$ from each equivalence class modulo $\Lambda$. 
\end{theorem}

\begin{example} In Example \ref{ex:basicexamples}, 
let $\liet=\liet^\ast=\bR$ with pairing given by multiplication, and let 
$\Lambda=\Xi=\bZ$. Let $\bfalpha=\{1,\ldots,1\}$ be the weight $1\in \Lambda^\ast=\bZ$, repeated $n$ times, defining the Verlinde sum $V_{\bfalpha}=V_n$. Thus $f_{\bfalpha}(X)=\big(1-e^{-2\pi \i X}\big)^{-n}$, and 
$\Box\bfalpha=(0,n)$. The affine hyperplane arrangement $\calA$ is the set $\bZ$ of integers, and has a single vertex (up to translation) $p=0$. The basis $B_p$ consists of just one element $\bfH$, which is given by the single hyperplane $H=\{0\}$.  For the chamber $\mathfrak{c}=(0,1)$, we obtain
\[ T_{\frakc,\bfH}(\lambda,\ell)(X)=e^{2\pi \i \lambda X} \frac{2\pi \i \ell X}{1-e^{2\pi \i \ell X}}.\]
Writing $z=2\pi \i X$, Szenes' formula hence shows that $V_n(\lambda,\ell)=\on{Ver}_n(\lambda,\ell)$ for all 
 $n\ge 1$ and all $\lambda,\ell$ such that $\lambda\in (-n,\ell)$, with 
\begin{equation}\label{eq:basicszenes}
\on{Ver}_n(\lambda,\ell)=
\CT_{z=0}\Big(e^{\lambda z} \frac{\ell z}{1-e^{\ell z}}\ \ \frac{1}{(1-e^{-z})^n}\Big).
\end{equation}
See \cite[Theorem 4]{ve:res} for a direct proof of this formula, which  may be reformulated as the generating function \eqref{eq:vergenfunction} stated in the introduction. 
%
%
In a similar way, Szenes' formula for Bernoulli series reduces to the fact that 
$B_n(\lambda)$ agrees with $\on{Ber}_n(\lambda)$ for $0<\lambda<1$, where $\on{Ber}_n(\lambda)$ is given by the generating series \eqref{eq:bergenfunction}. 
%
\end{example}

A more elaborate example will be given in Section \ref{sec:example} below.

\subsection{Relation between Bernoulli series and Verlinde sums}
\label{subsec:BernoulliAndVerlinde}
We briefly describe a Khovanskii-Pukhlikov-type formula, also due to Szenes (\cite[Proposition 4.9]{sz:res} with somewhat different notation), relating Verlinde sums to Bernoulli series.  See also \cite[Theorem 5]{ve:res} for discussion of the 1-dimensional case. 

We first need the analogue of Bernoulli series centred at vertices $p \ne 0$.  Let $p \in \Lambda \otimes \bQ \subset \t$ and $\bfalpha=(\alpha_1,...,\alpha_n)$ a list of weights.  Let $\widehat{\bfalpha}=(\widehat{\alpha}_1,...,\widehat{\alpha}_n)$ be the list of affine linear functions vanishing at $p$, with linear parts given by $\bfalpha$.  Define the Bernoulli-like sum
\[ B_{\widehat{\bfalpha}}(\lambda,\ell)=\sideset{}{'}\sum_{\xi \in \tfrac{1}{\ell}\Xi} \frac{e^{2\pi \i \pair{\lambda}{\xi}}}{\prod_{\widehat{\alpha}\in \widehat{\bfalpha}} 2\pi \i \widehat{\alpha}(\xi)},\]
where the prime next to the summation means to omit terms such that the denominator vanishes.  If $\ell=1$ and $p \in \Xi$, then this is related to the Bernoulli series $B_{\bfalpha}$ by a multiplicative factor.  Szenes \cite{sz:it,sz:res} gave a version of Theorem \ref{SzenesFormulaTheorem} for the Bernoulli series $B_{\widehat{\bfalpha}}$; the result has the same form, but with $f_{\ubfalpha}(X)$  replaced by 
$\prod_{\alpha \in \widehat{\bfalpha}} (2\pi \i \widehat{\alpha}(X))^{-1}$. 

For $\alpha \in \t^\ast$, let $\alpha(\partial)=\alpha(\partial_\lambda) \colon \Pol(\t^\ast) \rightarrow \Pol(\t^\ast)$ denote the directional derivative in the direction $\alpha$, a first-order linear differential operator acting on the space of polynomial functions in $\t^\ast$.  The assignment $\alpha \mapsto \alpha(\partial)$ extends multiplicatively to a map from the completion of the symmetric algebra $\Sym(\t^\ast)$ to infinite-order differential operators, acting on $\Pol(\t^\ast)$.  In other words, to each formal power series $\P(X)$ in $X \in \t$, we assign an infinite-order differential operator $\P(\partial)$.  More generally, if $\P(X)$ is the Taylor expansion of a function at $p \in \t$, we can define an infinite-order differential operator $\P(\partial)$, acting on the $\Pol(\t^\ast)$-module $e^{\pair{\lambda}{p}}\cdot \Pol(\t^\ast)$.  The equation
\[ \alpha(\partial_\lambda-p)e^{\pair{\lambda}{p}}g=e^{\pair{\lambda}{p}}\alpha(\partial_\lambda) g, \qquad \alpha \in \t^\ast,\]
implies that only finitely many terms of the infinite series $\P(\partial)$ act non-trivially on any fixed element of this module.

Let $\ubfalpha$ be a list of augmented weights.  Choose representatives $p \in \Lambda \otimes \bQ$ for $\vx(\A)/\Lambda$.  Let $\ubfalpha_p$ be the sublist of $\ubfalpha$ consisting of the $\ualpha =(\alpha,u)$ such that $ue^{-2\pi \i \pair{\alpha}{p}}=1$.  For each $\ualpha=(\alpha,u)$ in $\ubfalpha_p$, let $\widehat{\alpha}$ be the unique affine linear function vanishing at $p$ and with linear part $\alpha$.  Denote the corresponding list of affine linear functions by $\widehat{\bfalpha}_p$.

Let $\Td_{\ubfalpha,p}(X)$ denote the Taylor series at $p$ of
\[ \prod_{\ualpha \in \ubfalpha_p} \frac{2\pi \i \widehat{\alpha}(X)}{1-e^{-2\pi \i \widehat{\alpha}(X)}}.\]
The function  
\[ f_{\ubfalpha_p^{\mathsf{c}}}(X)=\prod_{(\alpha,u) \in \ubfalpha_p^{\mathsf{c}}} \big(1-ue^{-2\pi \i \pair{\alpha}{X}}\big)^{-1} \]
is holomorphic on a neighborhood of $p$, and we will use the same symbol $f_{\ubfalpha_p^{\mathsf{c}}}(X)$ to denote its Taylor series at this point.  By differentiating inside the constant term of the residue formula for Bernoulli series and comparing with the residue formula for Verlinde sums, one has 
\begin{equation} 
\label{eqn:TdOpGeneral}
\Ver_{(\ubfalpha;\gamma)}(\lambda,\ell)=\sum_{p \in \vx(\A)/\Lambda}\Td_{\ubfalpha,p}\big((2\pi \i)^{-1}\partial \big)f_{\ubfalpha_p^{\mathsf{c}}}\big((2\pi \i)^{-1}\partial \big)\Ber_{(\widehat{\bfalpha}_p;\gamma)}(\lambda,\ell).
\end{equation}
In case there is only a single vertex at $p=0$, equation \eqref{eqn:TdOpGeneral} simplifies:
\begin{equation} 
\label{eqn:TdOp}
\ell^{-n}\Ver_{(\bfalpha;\gamma)}(\lambda,\ell)=\Big(\Td_{\bfalpha}\big((2\pi \i \ell)^{-1}\partial \big)\Ber_{(\bfalpha;\gamma)}\Big)(\lambda/\ell), 
\end{equation}
where $n$ is number of elements in the list $\bfalpha$.  Equation \eqref{eqn:ClassLim} follows from \eqref{eqn:TdOp}.  As a simple example of \eqref{eqn:TdOp}, consider $\Ber_2(\lambda)$, $\Ver_2(\lambda,\ell)$ introduced in Example \ref{ex:basicexamples}.  One has
\[ \Big(1+\frac{1}{2\ell}\frac{\partial}{\partial \lambda}+\frac{1}{12\ell^2}\frac{\partial^2}{\partial \lambda^2}\Big)^2\Ber_2(\lambda)=-\frac{1}{2}\lambda^2+\frac{1}{2}\lambda-\frac{1}{12}-\frac{\lambda}{\ell}+\frac{1}{2\ell}-\frac{5}{12\ell^2}.\]
Making the substitution $\lambda \leadsto \lambda/\ell$, and then multiplying by $\ell^2$, one arrives at the formula for $\Ver_2(\lambda,\ell)$.

\subsection{Proof of Theorem \ref{th:quasipoly}}
\bigskip

\noindent We now explain how to deduce the quasi-polynomial behavior (Theorem \ref{th:quasipoly}) from Szenes' Theorem \ref{SzenesFormulaTheorem}.  Using 
Proposition \ref{prop:lowerdimensional}, and since $\#(\lieT_\ell\cap \lieT_{\Delta_0})$ is a polynomial in $\ell$, we may assume that $\sp_{\bR}\bfalpha=\liet^*$. 
Furthermore, by 
\ref{subsec:basic} \eqref{it:primitive}, we may assume that the linear parts $\alpha$ are primitive lattice vectors in $\Lambda^\ast$. We will show that each of the terms 
\begin{equation}\label{eq:szenes}
\iCT_{\bfH}\big(T_{\frakc,\bfH}(\lambda,\ell)f_{\ubfalpha}\big)
\end{equation} 
with $T_{\frakc,\bfH}(\lambda,\ell)(X)$ 
given by \eqref{eq:tcbeta} and $f_{\ubfalpha}(X)$ given by \eqref{eq:falpha}, is quasi-polynomial in $\lambda,\ell$. Write $\bfbeta=(\beta_1,\ldots,\beta_r)$. We have to compute the iterated constant term with respect to the coordinates $z_k=2\pi \i \beta_k(X)$. We have that 
\[ e^{2\pi \i \pair{\beta_k^0}{X}}=v_k^{-1} e^{z_k},\]
with the phase factor $v_k=e^{2\pi \i \beta_k(0)}$. 
The term 
\[ 
\prod_{\beta \in \bfbeta} \frac{2\pi \i \ell \beta(X)}{1-e^{2\pi \i \ell \pair{\beta^0}{X}}}
=\prod_{k=1}^r \frac{\ell z_k}{1-v_k^{-\ell} e^{\ell z_k}}
\]
in \eqref{eq:szenes}
is holomorphic in $z_1,\ldots,z_r$, and the coefficients of its power series expansion in the $z_k$ are quasi-polynomial functions of $\ell$. (Here we are using that the map $\ell\mapsto v_k^\ell$ is periodic in $\ell$, since $\beta_k(0)$ are rational.) Next, letting 
$\lambda_k$ be the components of $\lambda$ with respect to the basis $\beta^0_k$, 
and similarly for $\mu$, we see that
\[ e^{2\pi \i \pair{\lambda-\ell\mu}{X}}=(e^{-2\pi \i \pair{\mu}{p}})^\ell\ t_p^{\lambda}
\prod_{k=1}^r e^{(\lambda_k-\ell\mu_k)z_k},\]
with  $t_p=\exp(p)\in \lieT$. 
The phase factor $(e^{-2\pi \i \pair{\mu}{p}})^\ell\ t_p^{\lambda}$, 
is quasi-polynomial in $\lambda,\ell$, while $\prod_{k=1}^r e^{(\lambda_k-\ell\mu_k)z_k}$ 
is a holomorphic function of $z_1,\ldots,z_r$ whose power series coefficients are polynomial 
in $\lambda,\ell$. In summary, $T_{\mathfrak{c},\bfH}(\lambda,\ell)\circ \bfbeta^{-1}$, 
is a meromorphic function of $z_1,\ldots,z_r$ whose Laurent series coefficients are quasipolynomial 
in $\lambda,\ell$. Finally, $f_{\ubfalpha}\circ \bfbeta^{-1}$ is a meromorphic function of the $z_1,\ldots,z_r$ 
that does not depend on $\lambda,\ell$. Hence, \eqref{eq:szenes} is a quasi-polynomial function of $\lambda,\ell$.

\begin{remark}
If $\ubfalpha=(\bfalpha,\bfu)$ is such that all $u=1$, and $\bfalpha$ is unimodular (i.e., any vector space basis consisting of elements of $\bfalpha$ is also a basis of $\Lambda^*$), then there is only the zero vertex $p=0$, up to $\Lambda$-translation. Thus all $v$'s are equal to $1$ as well, and the argument above shows that $V_{\bfalpha}$ is not only quasi-polynomial on the region \eqref{eq:region}, but 
is in fact as a \emph{polynomial} on that region.
\end{remark}

\section{Example: Verlinde sum associated to $G=\SU(3)$}\label{sec:example}
\label{SzenesSU(3)Example}
In this Section, we work out the terms of the decomposition formula for the case of $G=\mathrm{SU}(3)$, continuing  Example \ref{ex:lie2}. For a wealth of examples of this type in the case of Bernoulli series, see \cite{bal:mul};  
for calculations of generalized Heaviside functions and partition functions associated to root systems, see \cite{bal:vol}. 

We begin with a general compact, simple, simply connected Lie group $G$. We fix the standard notation: $\lieT$ is a maximal torus, $W$ the corresponding Weyl group, $\Lambda\subset \liet$ is the integral (i.e., root) lattice, $\Xi=\Lambda^*$ is the weight lattice (using the basic inner product to identify $\liet$ and $\liet^\ast$). We denote by $\liet_+$ the closed positive Weyl chamber, corresponding to the set $\mf{R}_+$ of positive roots, and let $\liet_-=-\liet_+$ and $\mf{R}_-=-\mf{R}_+$. We are interested in the Verlinde sum $V_{\bfalpha}$ for $\bfalpha=\mf{R}_-$. 

Let $Z\subset T$ be the center of $G$. The function $t\mapsto t^\alpha$ defined by roots $\alpha$ descends to $T/Z$; hence the Verlinde sum can be written as a sum  
\[ V_{\bfalpha}(\lambda,\ell)=\Big(\sum_{\zeta\in Z}\zeta^\lambda\Big)
\sum_{[t]\in T_\ell/Z} \frac{t^\lambda}{\prod_{\bfalpha}1-t^{-\alpha}}.
\]
using a representative $t$ for each equivalence class $[t]$. The sum 
$\sum_{\zeta\in Z}\zeta^\lambda$ vanishes unless $\lambda$ is in the 
root lattice, in which case it is equal to $\# Z$. In particular, we see that 
$V_{\bfalpha}(\cdot,\ell)$ is supported on the root lattice. if $G$ is simply laced, then the root lattice is equal to $\Lambda$, and the prefactor is 
$\# Z\ \delta_\Lambda(\lambda)$. 

We now specialize to the case $G=\SU(3)$. Here $\Lambda$ has a basis $\alpha_1,\alpha_2$ given by the simple roots, while $\Lambda^*$ has the dual basis $\varpi_1,\varpi_2$ given by the fundamental weights.
It will be convenient to denote by $\alpha_3=\alpha_1+\alpha_2$ the highest root, so that $\bfalpha=\{-\alpha_1,-\alpha_2,-\alpha_3\}$.

The collection $\calS$ of admissible subspaces consists of subspaces spanned by subsets of roots,  together with their translates under the integral (i.e., root) lattice $\Xi^*=\Lambda$. Take the open chamber $\mathfrak{c}$  to be the the interior of the triangle spanned by $0,\alpha_2,\alpha_3$. An element 
\[ \gamma\in\mf{c}\] satisfies the genericity assumption from Section \ref{sec:decomp} if and only if it lies in the interior of an alcove from the Stiefel diagram;  for convenience, we will will take it to be in the interior of the fundamental Weyl alcove (the triangle spanned by $0,\varpi_1,\varpi_2$).

The decomposition formula 
is a sum $V_{\bfalpha}=\sum_\Delta V_{\bfalpha}^\Delta$ over admissible subspaces, where each $V_{\bfalpha}^\Delta$ 
is a convolution of a polynomial $\Ver_{(\alpha_\Delta,\gamma_\Delta)}$ and a partition function. 
We will discuss these terms separately, according to  the dimension of $\Delta$.

\subsection{Contribution from $\Delta=\liet^*$}
The leading term in the decomposition formula for $V_{\bfalpha}$ is the polynomial expression  $\Ver_{(\bfalpha;\gamma)}$ associated to the admissible subspace $\Delta=\t^\ast$.  
To compute $ \Ver_{(\bfalpha;\gamma)}$, we will use the Szenes formula.

The affine hyperplane arrangement $\calA$ is given by the set of affine root hyperplanes (the Stiefel diagram). Up to $\Lambda$-translation, it has three vertices given by the vertices $p=0,\varpi_1, \varpi_2$ of the Weyl alcove. 

Let $p \in \{0,\varpi_1,\varpi_2 \}$.  Using the ordering of root hyperplanes $H_{\alpha_1}<H_{\alpha_2}<H_{\alpha_3}$, we obtain an nbc basis 
$B_p=\{\bfH_{p,1},\bfH_{p,2}\}$ where $\bfH_{p,1},\bfH_{p,2}$ are defined by the linear functionals
\[ \bfbeta_{p,1}=\{\alpha_1-c_1,\alpha_2-c_2\}, \hspace{1cm} \bfbeta_{p,2}=\{\alpha_1-c_1,\alpha_3-c_3\},\]
respectively, with $c_k=\l\alpha_k,p\r \in \{0,1 \}$.  Explicitly $(c_1,c_2,c_3)=(0,0,0)$ for $p=0$, $(1,0,1)$ for $p=\varpi_1$, and $(0,1,1)$ for $p=\varpi_2$.  Note that $\Xi^\ast=\Lambda$, and 
\[ (\mathfrak{c}-\Box \bfbeta_{p,1}^{0})\cap \Xi^\ast=\{0\},\ \ 
(\mathfrak{c}-\Box \bfbeta_{p,2}^{0})\cap \Xi^\ast=\{-\alpha_1\};\]
in particular $\vol_{\Xi^\ast}(\Box \bfbeta^0_{p,k})=1$ for $k=1,2$. The Szenes formula for $V_{\bfalpha}$ reads as 
\[ V_{\bfalpha}=\sum_p \iCT_{\bfH_{p,1}}(T_{p,1} f_{\bfalpha})+\iCT_{\bfH_{p,2}}(T_{p,2} f_{\bfalpha})
\]
where 
\[ T_{p,1}(X)=\frac{e^{2\pi \i \pair{\lambda}{X}}(2\pi \i \ell)^2 (\pair{\alpha_1}{X}-c_1)(\pair{\alpha_2}{X}-c_2)}{ (1-e^{2\pi \i \ell \pair{\alpha_1}{X}})(1-e^{2\pi \i \ell \pair{\alpha_2}{X}});
}\]
\[ T_{p,2}(X)=\frac{e^{2\pi \i \pair{\ell \alpha_1+\lambda}{X}}(2\pi \i \ell)^2 (\pair{\alpha_1}{X}-c_1)(\pair{\alpha_3}{X}-c_3)}{ (1-e^{2\pi \i \ell \pair{\alpha_1}{X}})(1-e^{2\pi \i \ell \pair{\alpha_3}{X}});
}\] 

To compute the iterated residue, introduce variables $z_k=2\pi \i (\pair{\alpha_k}{X}-c_k)$, let $t_p=\exp(p) \in Z$, and write $\lambda=\mu_1\varpi_1+\mu_2\varpi_2$.  Using $\varpi_1=\tfrac{2}{3}\alpha_1+\tfrac{1}{3}\alpha_2$, $\varpi_2=\tfrac{1}{3}\alpha_1+\tfrac{2}{3}\alpha_2$ the first term reads as 
\[\iCT_{\bfH_{p,1}}(T_{p,1} f_{\bfalpha})= \CT_{z_1=0}\CT_{z_2=0}\left( 
\frac{t_p^\lambda e^{(\frac{2}{3}\mu_1+\frac{1}{3}\mu_2)z_1+(\frac{1}{3}\mu_1+\frac{2}{3}\mu_2)z_2} \,\ell^2 \,z_1 z_2}
{(1-e^{z_1})(1-e^{z_2})(1-e^{z_1+z_2})(1-e^{\ell z_1})(1-e^{\ell z_2})} 
\right).\] 
For the second term, one uses the variable $z_3=z_1+z_2$ in place of 
$z_2$, and finds 
\[\iCT_{\bfH_{p,2}}(T_{p,2} f_{\bfalpha})= \CT_{z_1=0}\CT_{z_3=0}\left( 
\frac{t_p^\lambda e^{(\frac{1}{3}\mu_1-\frac{1}{3}\mu_2)z_1+(\frac{1}{3}\mu_1+\frac{2}{3}\mu_2)z_3} e^{\ell z_1}\,\ell^2 \,z_1 z_3}
{(1-e^{z_1})(1-e^{z_3-z_1})(1-e^{z_3})(1-e^{\ell z_1})(1-e^{\ell z_3})} 
\right).\] 
Note that the terms for $p=0,\varpi_1,\varpi_2$ are the same except for the factor $t_p^\lambda$.  These elements $t_p$ are just the elements of the center 
$Z$ of $\SU(3)$, hence  $\sum_p t_p^\lambda=3\delta_\Lambda(\lambda)$. 
Carrying out the calculations of the constant terms with the help of \emph{Wolfram Alpha}, we arrive at the following expression:
\begin{align*}
\Ver_{(\bfalpha;\gamma)}(\lambda,\ell)&=\hh(1-\mu_1)\Big(\ell^2-(\mu_1+2\mu_2-3)\ell+(\mu_2-1)(\mu_1+\mu_2-2)\Big)\ \delta_\Lambda(\lambda).
\end{align*} 
This may also be written, 
\begin{equation}\label{eq:alternative}
 \Ver_{(\bfalpha;\gamma)}(\lambda,\ell)=-\hh (\mu_1-1)(\mu_2-\ell-1)(\mu_1+\mu_2-\ell-2)
\delta_\Lambda(\lambda)
\end{equation}
One may verify (as we did) that a different choice of nbc basis gives the same result.  

By the general theory, 
$V_{\bfalpha}(\lambda,\ell)$ coincides with $\Ver_{(\bfalpha;\gamma)}(\lambda,\ell)$
for all $(\lambda,\ell)\in\Lambda^*\times \bN$ with $\lambda$ in the region  $\ell \frakc-\Box \bfalpha=\ell \frakc+\Box \calR_+$. A simple calculation shows that this region is the interior of the hexagon spanned by the orbit of $0$ under the shifted Weyl group action 
\begin{equation}\label{eq:shiftweyl}
		w\bullet \lambda=w(\lambda-\tau)+\tau,\end{equation} 
where 
$\tau=\ell\varpi_2+\rho$ is the barycenter of the hexagonal region. 
One can check that the polynomial  is anti-invariant under this action: 
\[ \Ver_{(\bfalpha;\gamma)}(w\bullet \lambda,\ell)=(-1)^{\on{length}(w)}\Ver_{(\bfalpha;\gamma)}(\lambda,\ell).\]


\begin{remark}
Making the replacement $\mu_i \leadsto \ell \mu_i$ in \eqref{eq:alternative}, dividing by $\ell^3$ and taking the limit as $\ell \rightarrow \infty$, the polynomial in the expression above becomes
\[ -\tfrac{1}{2}\mu_1(\mu_2-1)(\mu_1+\mu_2-1).\]
This is $3$ times the polynomial (for the same chamber $\mathfrak{c}$) for the Bernoulli series $B_{\bf\alpha}(\lambda)$, as computed by Baldoni-Boysal-Vergne \cite[Equation (2.5.2)]{bal:mul}.  The factor of $3$ comes from the $\delta_\Lambda$ factor, and since
$\#(\Lambda^*/\Lambda)=\#Z=3$. 
\end{remark}

\subsection{Contributions with $\dim\Delta=1$. }
The $1$-dimensional admissible subspaces are of the form $\Delta=\xi+\bR \alpha$ where $\alpha \in \bfalpha$ and $\xi \in \Xi^\ast=\Lambda$.  The corresponding contribution  $V_{\bfalpha}^\Delta$ is a convolution 
of $\Ver_{(\bfalpha_\Delta;\gamma_\Delta)}$, where $\bfalpha_\Delta=\{\alpha\}$, 
with the partition function $P_{(\bfalpha_\Delta^{\mathsf{c}};\tau_\Delta)}$, where $\bfalpha_\Delta^{\mathsf{c}}$ are the two remaining roots
in $\bfalpha$. For the calculation below, we will find it convenient to work with the basis $\alpha_1,\alpha_2$ of $\liet$, 
so we will write $\lambda=\lambda_1\alpha_1+\lambda_2\alpha_2$ and only 
later switch to the basis $\varpi_1,\varpi_2$. 

We compute the contribution from $\Delta=\bR \alpha_1$; thus $\bfalpha_\Delta=\{-\alpha_1 \}$.  The chambers of $\Delta$ are the intervals 
$(k,k+1)\alpha_1$  for $k\in\bZ$; here the chamber of 
$\gamma_\Delta$ is the interval for $k=0$, while the polarizing vector $\tau_\Delta$ is a positive multiple of $-\varpi_2$.

By Proposition \ref{prop:lowerdimensional}, the restriction of $V_{\bfalpha_\Delta}$ to $\Delta$ is $\#(\lieT_\ell\cap \lieT_{\Delta})$ times a lower-dimensional Verlinde sum. In our case, $\#(\lieT_\ell\cap \lieT_{\Delta})=3\ell$, while the lower-dimensional Verlinde sum is
\[ \lambda\mapsto \sideset{}{'}\sum_{\zeta^\ell=1}\frac{\zeta^{\lambda_1}}{1-\zeta}.\]
This is nearly the same as $V_1$ defined in Section \ref{SubsectionDefandQuasi}, and one finds similarly that it is 
supported on $\Delta\cap\Lambda\subset \Delta\cap \Lambda^*$, 
and equals the polynomial $\lambda\mapsto \lambda_1-\tfrac{1}{2}(\ell+1)$
for $\lambda_1 \in (0,\ell+1)\cap \bZ$.  It follows that 
\[ \Ver_{(\bfalpha_\Delta;\gamma_\Delta)}(\lambda,\ell)=3\ell \Big(\lambda_1-\f{1}{2}(\ell+1)\Big)\delta_{\bZ}(\lambda_1)
\delta_0(\lambda_2).\]
As for  the partition function $P_{(\bfalpha_\Delta^{\mathsf{c}};\tau_\Delta)}$, note that the list 
$\bfalpha^{\mathsf{c}}_\Delta=\{-\alpha_2,-\alpha_3\}$ is polarized already.  
Hence, the partition function has support in the intersection of $\Lambda$ with the cone generated by $-\alpha_2,-\alpha_3$ (without shifts). 
In terms of the coordinates  $\lambda_1,\lambda_2$ it is given by
\[ P_{(\bfalpha_\Delta^{\mathsf{c}};\tau_\Delta)}(\lambda)=\begin{cases} 1 & \mbox{ if }\lambda_2\le \lambda_1 \le 0,\  \lambda_i \in \bZ, \\ 0 & \text{ otherwise.} \end{cases} \]
Taking the convolution, one obtains
\begin{equation}\label{eq:oldbasis}    V_{\bfalpha}^{\R\alpha_1}(\lambda,\ell)   =\tfrac{3}{2}\ell (1-\lambda_2)\big(2\lambda_1-\lambda_2-\ell-1\big)\,H(-\lambda_2)\,\delta_\Lambda(\lambda),\end{equation}
where $H$ is the Heaviside function supported on $[0,\infty)$. Finally,  switching to the coordinates $\mu_1,\mu_2$ defined by $\lambda=\mu_1\varpi_1+\mu_2\varpi_2$ we arrive at the following formula for the contribution from $\Delta=\bR \alpha_1$:
\begin{equation}
\label{eqn:Delta1}
V_{\bfalpha}^{\R\alpha_1}(\lambda,\ell)
=\frac{\ell(3-\mu_1-2\mu_2)(\mu_1-\ell-1)}{2}
H\big(-\f{\mu_1+2\mu_2}{3}\big) \, \delta_\Lambda(\lambda) .
\end{equation}
The contribution $V_{\bfalpha}^{\Delta}$ of general affine subspaces $\Delta$, is obtained from $V_{\bfalpha}^{\R\alpha_1}$ as follows: Let 
$\ca{L}\colon \liet\to \liet$ be the unique orientation preserving affine-linear transformation, taking 
$\R\alpha_1$ to $\Delta$, taking $-\varpi_2$ to a positive multiple of 
$\tau_\Delta$, and taking $(0,1)\alpha_1$ to the chamber of $\Delta$ containing $\gamma_\Delta$. For each $\ell$, let $\ca{L}_\ell(\lambda)=\ell\ca{L}(\lambda/\ell)$ be the corresponding transformation at level $\ell$. Let $\sigma_\Delta$ be the shift vector (cf.~ \eqref{eq:sigmadelta}), given as minus the sum of roots in $\bfalpha_\Delta^{\mathsf{c}}$ that are 
not polarized, and let  $\pm 1$ be the determined by the number of sign changes. (It is easy to see that one gets $-1$ if $\Delta$ is parallel to $\alpha_3$, equal to $+1$ otherwise.). 
\[ V_{\bfalpha}^{\Delta}(\lambda,\ell)=\pm 
V_{\bfalpha}^{\R\alpha_1}(\ca{L}_\ell^{-1}(\lambda-\sigma_\Delta),\ell).\]



\subsection{Contributions with $\dim\Delta=0$}
The $0$-dimensional admissible subspaces are given by $\Delta=\{\xi\}$ for 
$\xi \in \Xi^\ast=\Lambda$.  Consider the contribution of a given $\xi$, 
Recall that $\gamma_\Delta=\xi$ and $\tau_\Delta=\xi-\gamma$. 
The term $V^{\{\xi\}}(\cdot,\ell)$ is the convolution of $\delta_{\ell\xi}$ 
with a partition function $P_{\bfalpha_\Delta,\tau_\Delta}$. Convolution with 
$\delta_{\ell\xi}$ amounts to shifting the argument $\lambda$ by $\ell\xi$, but the partition function depends on the polarization. 

Since we took $\gamma$ to be in the interior of the Weyl alcove, adding $-\gamma$ to $\xi$ shifts it into the interior of one of the Weyl chambers. 

Let us now first consider the case $\{\xi\}=0$. Then $\tau_\Delta=-\gamma$, so that $\bfalpha=\mf{R}_-$ is already polarized. 
Let $\Lambda^*\to \Z,\ \lambda\mapsto \mf{P}(\lambda)$ be the Kostant partition function, given as  
the number of ways of writing $\mu$ as a linear combination of elements of $\mf{R}_+$ with  non-negative coefficients. Writing $\lambda=\mu_1\varpi_1+\mu_2\varpi_2$, one has that 
\[ \mf{P}(\mu)=\begin{cases}
\f{1}{3}(\mu_2-\mu_1)\ \delta_\Lambda(\lambda) & \mu_2\ge \mu_1\ge -\hh \mu_2\\
\f{1}{3}(\mu_1-\mu_2)\ \delta_\Lambda(\lambda) & \mu_1\ge \mu_2\ge -\hh \mu_1\\
0 & \mbox{ otherwise}
\end{cases}
\]
Then 
\[ V^{\{0\}}=P_{\bfalpha_{\{0\}},\tau_{\{0\}}}(\lambda)=\mf{P}(-\lambda).\]
More generally, for $\xi\in \Lambda\cap\t_-$, we obtain a shifted version, 
\[ V^{\{\xi\}}=\mf{P}(\ell\xi-\lambda).\]

For more general $\xi\in\Lambda^*$, we have that   $\tau_\Delta\in w\t_-$ for a unique $w\in W$, also characterized as the shortest Weyl group element with $w^{-1}\xi\in\t_-$. Equivalently, $w$ is the unique element such that
 $\l\alpha,w^{-1} \tau_\Delta\r>0$ for all negative roots. We hence see that 
 \[ \bfalpha^{\on{pol}}=
 w\mf{R}_- ,\ \ \bfalpha_+ =\mf{R}_-\cap w\mf{R}_-,\ \ \bfalpha_- =\mf{R}_-\cap w\mf{R}_+.\]
As is well-known, the set $\mf{R}_-\cap w\mf{R}_+$ has cardinality $l(w)$, and 
 \[ \sigma:=-\sum_{\alpha\in\bfalpha_-}\alpha=\rho-w\rho.\] 

 We hence obtain 
 \[ P_{\bfalpha,\tau_\Delta}(\lambda)=(-1)^{\on{length}(w)}  
 P_{\bfalpha^{\on{pol}},\tau_\Delta}(\lambda-\rho+w\rho).\]

Observe that $\mf{P}(-\mu)$ is the 
number of ways of writing $\mu$ as a linear combination of elements of 
$\mf{R}_-$ with  non-negative coefficients, and is thus also the number of ways of writing 
$w\mu$ as a linear combination of elements of 
$w\mf{R}_-=\bfalpha^{\on{pol}}$ with  non-negative coefficients. That is, 
$\mf{P}(-\mu)=P_{\bfalpha^{\on{pol}},\tau_\Delta}(w\mu)$. This shows 
\[ P_{\bfalpha,\tau_\Delta}(\lambda)=(-1)^{l(w)}  
 \mf{P}(-w^{-1}(\lambda-\rho)-\rho).\]
Finally, convolution with $\delta_{\ell\xi}$ amounts to replacing $\lambda$ with $\lambda-\ell\xi$. We conclude that the contribution from $\Delta=\{\xi\}$, 
with $\xi\in w\t_-$, is 
\[  V^\Delta_{\bfalpha}(\lambda,\ell)= (-1)^{l(w)}  
 \mf{P}(w^{-1}(\ell\xi-\lambda+\rho)-\rho).\]

\section{Equivariant Bernoulli series and Verlinde sums.}
In this section we define Verlinde sums and partition functions with additional equivariant parameters. The decomposition formula for these sums is relevant for applications to localization formulas in differential geometry, with the parameters as the recipients for `curvatures'. Note that the idea of using curvatures variables in the various partition functions is also used in Vergne's articles \cite{ver:pos,ver:for}.

\subsection{Equivariant Bernoulli series.}
It will be convenient to consider the lists of weights $\bfalpha=\{\alpha_1,\ldots,\alpha_n\}$ in $\Lambda^*$ as the weights of a unitary  $\lieT$-representation 
\[ A\colon \lieT\to \on{U}(n),\ t\mapsto A(t), \] 
where $A(t)e_i=t^{\alpha_i}\,e_i$. We will use the same letter to denote the infinitesimal representation $A\colon \liet\to \mf{u}(n)$. With this notation, the Bernoulli series  can be written as a sum $B_{A}(\lambda)={\sum}'_{\xi\in \Xi} e^{2\pi \i \l\lambda,\xi\r} \det(A(\xi))^{-1}$.
Letting $\lieH\subset \U(n)$ be the Lie group of transformations commuting with all $A(t)$, and $\lieh$ its Lie algebra, we define the `$\lieH$-equivariant' multiple Bernoulli sum as the following (generalized) 
function of $\lambda$, for $X\in \lieh$ sufficiently small:
\begin{equation}\label{eq:equiber}
  B_A(\lambda;\,X)=\sideset{}{'}\sum_{\xi\in \Xi} \frac{e^{2\pi \i \l\lambda,\xi\r}}{\det(A(\xi)+X)}.
 \end{equation}
Here ${\sum}'$ signifies, as before, a sum over all $\xi\in\Xi$ such that 
$\det(A(\xi))\neq 0$. A vector $\tau\in \liet$ is polarizing for the list $\bfalpha$ if 
and only if $\det A(\tau)\neq 0$, and in this case we may define 
\begin{equation}\label{eq:equiheav}
  H_{A,\tau}(\lambda;\,X)=\lim_{\epsilon\to 0^+}\int_{\xi\in\t} \frac{e^{2\pi \i \l\lambda,\xi\r}}{\det(A(\xi-\i \epsilon\tau)+X)}\ \d\xi.
  \end{equation}
Both $X\mapsto B_A(\lambda;\,X)$ and $X\mapsto H_A(\lambda;\,X)$ depend analytically on $X$, for $X\in\lieh$ in a sufficiently small neighborhood of zero, in the sense that the integral against test functions is analytic. But for our purposes, it will be quite enough to treat $X$ as a formal parameter, and thus to consider the Taylor expansions. 

Consider the decomposition of the symmetric algebra as sum 
over multi-indices $J=(j_1,\ldots,j_n)$ with $j_k\ge 0$,
\[ S(\bC^n)=\bigoplus_J S^J(\bC^n),\]
where $S^J(\bC^n)=S^{j_1}(\bC)\otimes \cdots \otimes S^{j_n}(\bC)$. Since $A(t)$ are diagonal matrices, we obtain representations $A^J$ of $\lieT$ on $S^J(\bC^n)$. 


%
\begin{lemma} \label{lem:b}
We have the expansion, for $X\in \mf{u}(1)^n$, 
\[ \det(A(\xi)+X)^{-1}=\sum_J f_J(X)\,  \det(A^{J}(\xi))^{-1}.\]
Here the sum is over multi-indices $J=(j_1,\ldots,j_n)$ with $j_k>0$, and the Taylor expansion of $f_J(X)$ starts with terms of order $\ge |J|-n$.  
		%
\end{lemma}
\begin{proof}

Since both $A(t)$ and $X$ are diagonal it is enough to consider the case $n=1$. But for $n=1$, the claim just follows from $(t+z)^{-1}=\sum_{j=0}^\infty z^j (-1)^j t^{-j-1}$.  
\end{proof}
Using the lemma, we have the expansion 
\begin{equation}\label{eq:bernoulliexpansion}
 B_A(\lambda;\,X)=\sum_{J} f_J(X)\, B_{A^J}(\lambda)\end{equation}
for the Bernoulli series, and similarly 
\[ H_{A,\tau}(\lambda;X)=\sum_J f_J(X) H_{A^J,\tau}(\lambda)\]
(with the same $f_J$).  In particular, the Taylor coefficients of $B_A(\lambda;X)$ are 
linear combinations of Bernoulli series for $A^J$. Since only $J$ with $j_k>0$ appear, the list $\bfalpha^J$ of weights appearing in the representation $A^J$ is obtained from the original list  $\bfalpha$ of weights appearing in $A$ by increasing some of the multiplicities. 
Hence, $\calS(\bfalpha^J)=\calS(\bfalpha)$. Consequently,  the the Taylor coefficients of $B_A(\lambda;X)$ are  supported on the union of $\Xi^\ast$-translates of 
$\Delta_0$, and are polynomial on each chamber of $\Delta_0$ and its translates.
A similar discussion applies to $H_{A,\tau}$. 

As  in Section \ref{sec:decomp}, we fix an integral inner product on $\liet$, thus identifying $\liet^*\cong \liet$, and we pick $\gamma\in \liet$, with the property that for all  $\Delta\in\calS$, the orthogonal projection $\gamma_\Delta$ onto $\Delta$ lies in a chamber for $\bfalpha_\Delta$, and $\tau_\Delta=\gamma_\Delta-\gamma$ is polarizing for $\bfalpha_\Delta^{\mathsf{c}}$. 
Consider the orthogonal decomposition $\bC^n=\bC^n_\Delta\oplus (\bC^n)_\Delta^\perp$, 
where  the first summand is the sum of the coordinate lines $\C e_j$ with $\alpha_j\in\bfalpha_\Delta$, and similarly the second summand is a 
sum of coordinate lines with $\alpha_j\in \bfalpha_\Delta^{\mathsf{c}}$. 
The two summands are subrepresentations $A_\Delta$, $A_\Delta^{\mathsf{c}}$ of $\lieT$; the summand $\bC^n_\Delta$ is the subspace on which the subtorus $\lieT_\Delta\subset \lieT $ acts trivially.

The $\lieH$-action preserves this decomposition. We hence obtain 
an $\lieH$-equivariant Bernoulli series 
$B_{A_\Delta}(\lambda;\,X)$, whose Taylor coefficients are polynomial on each 
chamber of $\Delta$. We define $\on{Ber}_{(A_\Delta;\gamma_\Delta)}(\lambda;\,X)$
with equivariant variables $X\in\lieh$, by requiring that its Taylor coefficients are generalized functions supported on $\Delta$, given by 
polynomials on $\Delta$, and agreeing with the Taylor coefficients of $B_{A_\Delta}(\lambda;\,X)$ on the chamber containing $\gamma_\Delta$. 
\begin{proposition}[Equivariant Boysal-Vergne decomposition formula] 
\label{prop:equiber}
For generic choices of $\gamma$ as above, the equivariant Bernoulli series decomposes as 	 
\begin{equation}\label{eq:boysalvergne1}
B_A(X)=\sum_{\Delta\in \calS} \on{Ber}_{(A_\Delta;\gamma_\Delta)}(X)\star 
H_{(A_\Delta^{\mathsf{c}},\tau_\Delta)}(X),
\end{equation}
(as an equality of formal power series in $X$, with coefficients that are 
functions of $\lambda\in\liet^*$). 
\end{proposition}
\begin{proof}
By $\lieH$-invariance, it suffices to prove the equality of both sides over $\mf{u}(1)^n$. For $X\in\mf{u}(1)^n$, we have the expansions
\[  \on{Ber}_{(A_\Delta;\gamma_\Delta)}(\lambda;\,X)=\sum_{J'} f_{J'}^\Delta(X) \on{Ber}_{(A_\Delta^{J'};\gamma_\Delta)}(\lambda),\]
where the $f_{J'}^\Delta(X)$ are defined similar to the $f_J(X)$, but using only the basis elements belonging to $\C^n_\Delta$. Similarly, 
\[ H_{(A_\Delta^{\mathsf{c}},\tau_\Delta)}(\lambda;X)=\sum_{ J''} f_{J''}^{\Delta,\mathsf{c}}(X)
H_{(A^{J''}_{\Delta,\mathsf{c}},\tau_\Delta)}(\lambda)\]
where $f_{J''}^{\Delta,\mathsf{c}}(X)$ are defined in terms of the representation on $(\C^n_\Delta)^\perp$. Taking the convolution of these expressions, and using,  
\[ f_J(X)=f_{J'}^\Delta(X)f_{J''}^{\Delta,\mathsf{c}}(X)\]
for $J=J'\cup J''$, we obtain 
\[ \on{Ber}_{(A_\Delta;\gamma_\Delta)}(X)\star 
H_{(A_\Delta^{\mathsf{c}},\tau_\Delta)}(X)=
\sum_J f_J(X) 
\on{Ber}_{(A_\Delta^{J'};\gamma_\Delta)}\star H_{(A^{J''}_{\Delta,\mathsf{c}},\tau_\Delta)}.
\]
Summing over all $\Delta$, and using the non-equivariant Boysal-Vergne decomposition formula we recover \eqref{eq:bernoulliexpansion}.
\end{proof}

\subsection{Equivariant Verlinde sums.}
should one 
For the Verlinde sums, we have the $\lieT$-representation $A$ as in the previous section, as well as possibly an additional diagonal matrix $U\in \on{U}(n)$ such that some power of $U$ is the identity. We denote by  $\ul{A}=(A,U)$ the `augmented representation'; it determines a list $\ubfalpha$ of augmented weights $(\alpha,u)$  for the 1-dimensional weight spaces. 

Take $\lieH\subset \on{U}(n)$ to be the subgroup of transformations commuting with $U$ and with all $A(t)$, and $\lieh\subset \mf{u}(n)$ its Lie algebra. As before, $\lieH$ contains $\U(1)^n$. For $X\in \lieh$ sufficiently small, put 
\begin{equation}\label{eq:equiver}
  V_{\ul{A}}(\lambda,\ell;\,X)=\sideset{}{'}\sum_{t \in \lieT_\ell} \frac{t^{\lambda}}{\det(1-U\,A(t^{-1})\exp(-X))}\end{equation}
where the sum is over all $t\in \lieT_\ell$ such that $U\,A(t^{-1})$ has no eigenvalue equal to $1$. This is a is a well-defined $\lieH$-invariant analytic function of $X$, for $X$ sufficiently small, which reduces to the Verlinde sum $V_{\ubfalpha}(\lambda,\ell)$ when $X=0$. 
We also have an equivariant version of the generalized partition function, for a 
given polarizing vector $\tau\in\liet$ such that $\l\alpha,\tau\r\neq 0$ for all weights $\alpha$:
\begin{equation}\label{eq:equipar}
 P_{(\ul A,\tau)}(\lambda;\,X)=\lim_{\epsilon\to 0^+}\int_{\lieT}  \frac{t^\lambda}{\det(1-U\,A(t^{-1}\exp(\i \epsilon\tau)) \exp(-X)) }\ \d t.
 \end{equation}
As in the case of the Bernoulli series, we will consider $X$ as a `formal variable', hence we will work with the Taylor expansions with respect to $X$. 

For multi-indices $J=(j_1,\ldots,j_n)$ with $j_k>0$, 
we obtain  $\lieT$-representations $A^J$ 
on $S^J(\bC^n)$, with commuting endomorphisms $U^J$.  

\begin{lemma}\label{lem:v} For $X\in \mf{u}(1)^n$, we have that 
		\[ \big(\det(1-U A(t^{-1})\exp(-X))\big)^{-1}= \sum_J g_J(X) 
		 \det(1-U^J A^J(t^{-1}))^{-1},\]
		where the sum is over multi-indices $J=(j_1,\ldots,j_n)$ with $j_k>0$, and where the Taylor expansion of $g_J(X)$ starts with terms of order $|J|-n$. 
\end{lemma}
\begin{proof}
Similar to the proof of Lemma \ref{lem:b}, it suffices to consider the case $n=1$, where it follows from 
\[ (1-a^{-1}\ e^{-z})^{-1}=\sum_{j=0}^\infty (-1)^j (1-e^{-z})^j (1-a^{-1})^{-j-1}.\qedhere\]
\end{proof}
It follows that
\begin{equation}\label{eq:verlindetaylor}
 V_{\ul A}(\lambda,\ell;\,X)=\sum_J g_J(X)\ V_{\ul{A} ^J}(\lambda,\ell),\end{equation}
and similarly for the partition function $P_{\ul A}(\lambda;\,X)$, with the same coefficients $g_J(X)$. By a discussion parallel to that for Bernoulli series, the Taylor 
coefficients of $V_{\ul A}(\lambda,\ell;\,X)$ are quasi-polynomial on 
the \emph{same} regions (e.g., \eqref{eq:region}) as for $X=0$. 

Once again, we fix an element $\gamma\in\liet$ that is generic with respect to $\calS$.  
For $\Delta\in \calS$  we obtain an equivariant  Verlinde sum $V_{\ul{A}_\Delta}(X)$, and a $\lieH$-equivariant function 
\[  \Ver_{(\ul{A}_\Delta;\gamma_\Delta)}(\lambda,\ell;\,X)\] 
whose Taylor coefficients 
are quasi-polynomial on $\{(\lambda,\ell)\colon \lambda\in \Delta\}$ 
and agrees with $V_{\ul{A}_\Delta}(\lambda,\ell;X)$ whenever  $\frac{1}{\ell}\lambda$ lies in the same chamber as $\gamma_\Delta$. 
The decomposition formula extends to the setting with parameters: 
\begin{proposition}[Equivariant decomposition formula for Verlinde sums]
\label{eq:equiver}
\begin{equation}\label{decompformal}
 V_{\ul A}(X)=
\sum_{\Delta\in\calS}   \Ver_{(\ul{A}_\Delta,\gamma_\Delta)}(X)\star 
P_{(\ul{A}_\Delta^{\mathsf{c}},\tau_\Delta)} (X)
\end{equation}
(convolution of functions on $\Lambda^\ast$, for fixed $\ell$). 
\end{proposition}
\begin{proof}[Outline of proof]
The proof is parallel to that of Proposition \ref{prop:equiber}, hence we will be brief.
By $\lieH$-invariance, it suffices to prove the identity for $X\in\mf{u}(1)^n$.  
The terms on the right hand side have expansions
\[  \Ver_{(\ul{A}_\Delta;\gamma_\Delta)}(X)
=\sum_{J'} g_{J'}^\Delta(X) \ \Ver_{(\ul{A}^{J'}_\Delta;\gamma_\Delta)},\]
and 
\[ P_{(A_\Delta^{\mathsf{c}},\tau_\Delta)}(X)=\sum_{ J''} g_{J''}^{\Delta,\mathsf{c}}(X)
P_{(A^{J''}_{\Delta,\mathsf{c}},\tau_\Delta)}.\]
Taking their convolution product, and using
\[ g_{J'}^\Delta(X)g_{J''}^{\Delta,\mathsf{c}}(X)= g_J(X)\]
one may carry out the summation over $\Delta$ by using the non-equivariant decomposition formula for Verlinde sums. The resulting expression is the 
expansion for $V_{\ul A}(X)$. 
\end{proof}

\subsection{Differential form valued Bernoulli series and Verlinde sums.} \label{SubsectionDiffFormValuedPartition}
%

Suppose now that $E\to M$ is a $\lieT$-equivariant Hermitian vector bundle over a connected base, where the action on the base $M$ is trivial. The $T$-action will be denoted $A_E\colon \lieT\to \on{Aut}(E)$; it may be regarded as a family of unitary $\lieT$-representations, smoothly labeled by the points of $M$. Letting $R_E\in \Omega^2(M,\on{End}(E))$ be the curvature of a $\lieT$-invariant connection on $E$, and 
$\Eul(E,\cdot)$ the corresponding $\lieT$-equivariant Euler form, we have for all $\xi\in\liet$
\[ \Eul(E,-2\pi \i \xi)=\det\big(A_E(\xi)+\tfrac{\i}{2\pi} R_E\big)\in \Om(M).\]
The expression 
\[ B_E(\lambda)=\sideset{}{'}\sum_{\xi\in \Xi^\ast} \f{e^{2\pi \i \l \lambda,\xi\r}}{\Eul(E,-2\pi \i \xi)}\in \Omega(M)\]
appears in the fixed point formula for Duistermaat-Heckman measures of Hamiltonian loop group spaces (see \cite{loi:nor}). It may be regarded as the characteristic form corresponding to an equivariant Bernoulli series. To see this more clearly, 
note that by rigidity of actions of compact Lie groups, the fibers of $E$ are 
 equivariantly isomorphic to a fixed $\lieT$-representation $t\mapsto A(t)$ on some $\bC^n$, which we may take to be a representation by diagonal matrices. 
The associated frame bundle $\on{Fr}(E)$, with fibers the $\lieT$-equivariant unitary isomorphisms from fibers of $E$ to $\bC^n$, has structure group $\lieH$ the unitary automorphisms commuting with $\lieT$. The Chern-Weil map takes $\lieH$-invariant formal power series on $\lieh$ to differential forms on the base. 
Applying the Chern-Weil map to $B_A(\lambda,X)$ we obtain the closed differential form $B_E(\lambda)$. As an $\Omega(M)$-valued  function of $\lambda$, it is polynomial on chambers of $\calS$. Given $\Delta\in\calS$, and a generic choice of $\gamma\in\liet$ as in previous sections, 
we define $\Omega(M)$-valued generalized functions of $\lambda$
\begin{equation}\label{eq:berh}
 \on{Ber}_{(E_\Delta;\gamma_\Delta)}(\lambda),\ \ H_{(E_{\Delta}^\perp,\tau_\Delta)}(\lambda)\end{equation}
where $E_\Delta\subset E$ is the subbundle fixed by $T_\Delta$, and $E_{\Delta}^\perp$ its orthogonal complement in $E$. Both are associated bundles to $\on{Fr}(E)$, and the differential forms \eqref{eq:berh} are obtained by applying the Chern-Weil map to the corresponding $\lieH$-equivariant invariant 
functions. We obtain the form-valued decomposition formula, 
\[  B_E=\sum_{\Delta\in\calS} \on{Ber}_{(E_\Delta;\gamma_\Delta)}\star H_{(E_\Delta^\perp,\tau_\Delta)}.\]

For the Verlinde sums, we allow for a slightly more general setting where $E$ comes with a  unitary automorphism $U_E$, fixing the base and commuting with $A_E$, and such that some power of $U_E$ is the identity. We write $\ol{A}_E=(A_E,U_E)$.
By rigidity of actions of compact groups, the fibers of $E$ are isomorphic to $\C^n$ with a given $\lieT$-action by diagonal matrices and an additional unitary automorphism $U$, also by diagonal matrices. Take $\on{Fr}(E)$ to be the 
corresponding frame bundle, consisting of unitary maps from fibers of $E$ to 
$\C^n$ intertwining $A_E,U_E$ with $A,U$, respectively; it is a principal $\lieH$-bundle, where  $\lieH\subset \on{U}(n)$ is the group of unitary automorphism commuting with both $A$ and $U$. 

Choose a principal connection on $\on{Fr}(E)$, and let $R_E\in \Omega^2(M,\End(E))$ be the curvature form of the resulting linear connection on $E$. Then 
\[ \calDC(E,t)=\det(1-U_E\,A_E(t)^{-1}e^{-\frac{\i}{2\pi} R_E})\in \Omega(M),\] 
for $t\in \lieT$
is the characteristic form corresponding to $\det(1-U\,A(t)^{-1}\exp(X))$. The expression
\[ V_{E}(\lambda,\ell):=\sideset{}{'}\sum_{t \in \lieT_\ell} \frac{t^{\lambda}}{ \calDC(E,t)}\in \Omega(M)\]
arises from the localization formula for the quantization of Hamiltonian loop group spaces \cite{al:fi,me:twi}; it is also obtained by applying the Chern-Weil map to 
the equivariant Verlinde sum $V_A(\lambda,\ell;X)$.  The decomposition formula for these differential form-valued Verlinde sums reads as 
\begin{equation}\label{eq:decompositionformula2}
	V_E=\sum_{\Delta \in \calS} \Ver_{(E_\Delta;\gamma_\Delta)}\star P_{(E_\Delta^\perp,\tau_\Delta)},
	\end{equation}
for generic choice of $\gamma\in\liet$. Here $\Ver_{(E_\Delta;\gamma_\Delta)}$ and $P_{(E_\Delta^\perp,\tau_\Delta)}$ are obtained by applying the Chern-Weil map to the corresponding $\lieH$-equivariant functions.

\bigskip\bigskip

\begin{appendix}
\section{Functions on lattices}\label{app:lattices}
Let $\Lambda$ be a lattice, with dual lattice $\Lambda^\ast=\Hom(\Lambda,\bZ)$. 
We denote by $\liet=\Lambda\otimes_{\bZ}\bR$ the vector space spanned by 
$\Lambda$ and by  $\lieT=\liet/\Lambda$ the torus.  

\subsection{Functions on $\Lambda^\ast$.}\label{appsubsec:functions} 
Consider the vector space $\Map(\Lambda^\ast,\bC)$ of functions $f\colon \Lambda^\ast\to \bC$, and its subspace $\Map_0(\Lambda^\ast,\bC)$ of functions of compact (i.e., finite) support. 
Every subset $S\subset \Lambda^\ast$ defines a \emph{delta-function}
\[\delta_S\in \Map(\Lambda^\ast,\bC),\ \ \delta_S(\lambda)=
\begin{cases}
1 &\colon  \lambda\in S\\ 0 &\colon \lambda\not\in S
\end{cases}\]

Every $t\in \lieT $ defines an evaluation function 
\[ \mathsf{e}_t\in \Map(\Lambda^\ast,\bC),\ \ \mathsf{e}_t(\lambda)= t^\lambda.
\]

\subsection{Push-forward, pull-back.} \label{appsubsec:pushpull} 
Given a lattice homomorphism $\phi\colon (\Lambda')^\ast\to \Lambda^\ast$ (dual to some morphism $\chi\colon \Lambda\to \Lambda'$)  we define the usual pull-back map $\phi^*\colon \Map(\Lambda^\ast,\bC)\to \Map((\Lambda')^\ast,\bC)$, as well as the  push-forward
\[ \phi_*\colon 
\Map_0((\Lambda')^\ast,\bC)\to \Map_0(\Lambda^\ast,\bC)\] 
by duality to the pull-back under $\chi$. Explicitly, 
$(\phi_* f)(\lambda)=\sum_{\lambda'\in \phi^{-1}(\lambda)} f(\lambda')$.
\subsection{Convolution.}\label{appsubsec:convol} 
As a special case, letting  $\mathrm{Add}\colon \Lambda^\ast\times \Lambda^\ast \to \Lambda^\ast$ be the addition (with dual map $\Lambda\to \Lambda\times \Lambda$ the diagonal inclusion), we define the convolution 
\[f_1\star f_2=\mathrm{Add}_*(f_1\otimes f_2).\]
Thus, $(f_1\star f_2)(\lambda)=\sum_{\lambda_1+\lambda_2=\lambda} f_1(\lambda_1)f_2(\lambda_2)$.
More generally, $f_1*f_2$ is defined even for functions of non-compact support, provided $\mathrm{Add}$ restricts to a proper map on the support of $f_1 * f_2$. 
Some basic properties of the convolution are 
\begin{equation} (\mathsf{e}_t \ f)\star(\mathsf{e}_t \ g)= \mathsf{e}_t \ (f \star g),
\end{equation}
\begin{equation}
(\phi^*f \star g)=\phi^*(f \star\phi_* g). \end{equation}
Convolution with $\delta_\mu$ for $\mu\in \Lambda^\ast$ acts as a translation: $(\delta_\mu\star f)(\lambda)=f(\lambda-\mu)$.

\subsection{Finite difference operators.} \label{appsubsec:findif} 
For any pair $\ul\beta=(\beta,v)$, where $\beta\in \Lambda^*$ and $v\in \bC$ (usually taken to be a root of unity), define the following finite difference operator on the space of functions $f\colon \Lambda^\ast\to \bC$:
\begin{equation}\label{eq:finitedifference}
 (\nabla_{\ul\beta} f)(\lambda)=f(\lambda)-v\, f(\lambda-\beta).\end{equation}
The finite difference operators for any two such functions $\beta_1,\beta_2$ commute. Under convolution of functions 
\begin{equation}
\nabla_\beta(f\star g)=(\nabla_\beta f)\star g=f\star (\nabla_\beta g)
\end{equation}

%
\subsection{Finite Fourier transform.}\label{appsubsec:finitefouriertransform}
Suppose $\Xi\subset \liet$ is a lattice containing $\Lambda$. Then 
$\Xi/\Lambda$ is a finite group, with dual group $\Lambda^*/\Xi^*$. 
If $f\in \Map(\Lambda^*,\bC)$ is $\Xi^*$-periodic, it has a finite Fourier expansion
\[ f(\lambda)=\sum_{t\in \Lambda/\Xi} f^\vee(t)\  t^\lambda;\ \ \  f^\vee(t)=\frac{1}{\# (\Xi/\Lambda)} \sum_{ \Lambda^*/\Xi^*}\ f(\lambda)\, t^{-\lambda}.\]
We have that $(\nabla_{\ul\beta}f)^\vee(t)=(1-v\,t^{-\beta})\,f^\vee(t)$.

\subsection{Poisson summation formula.}\label{appsubsec:poisson} 
Given a rational subspace $\lieh\subset \liet$, the (possibly disconnected) group $\exp(\Xi+\lieh)\subset \lieT $ has a normalized Haar measure. By pull-back, this defines a 
measure on $\Xi+\lieh\subset \liet$, which in turn pushes forward to a delta-measure 
\[ \delta_{\Xi+\lieh} \in \calD'(\liet)\]
supported on $\Xi+\lieh$. We may write this distribution as a sum $\sum_U \delta_U$
of delta-measures supported on affine subspaces $\xi+\lieh$, where $\xi$ ranges over representatives of $\Xi/\lieh$. We have the following version of the Poisson summation formula:

\begin{proposition}
	\label{PoissonSummation}
	Let $\Xi \supset \Lambda$ be full rank lattices in $\liet$, and let $\lieh \subset \liet$ be a rational subspace.  Let $dX$ be Lebesgue measure on $\liet$, normalized such that the induced measure on $\lieT=\liet/\Lambda$ is normalized Haar measure.  Then we have an equality of distributions on $\liet$, 
	\[ \Big(\sum_{\nu \in \Xi^\ast \cap \ann(\lieh)} e^{2\pi \i\pair{\nu}{X}}\Big)\ dX=
	\delta_{\Xi+\lieh}.\]
\end{proposition}
\begin{proof}
	Both sides are  $\Lambda$-invariant, hence  are pullbacks of distributions on $\lieT=\liet/\Lambda$.  The desired identity is equivalent to the equality of distributions on $\lieT$, 
	\[ \sum_{\nu \in \Xi^\ast \cap \ann(\lieh)} t^\nu \, dt=\delta_{\exp(\Xi+\lieh)}\]
	where $d t$ is the normalized Haar measure on $\lieT$ and $\delta_{\exp(\Xi+\lieh)}$ is the push-forward of normalized Haar measure on $\exp(\Xi+\lieh)$. To compare the Fourier coefficients of these two distributions, integrate against the function $t\mapsto t^{-\lambda}$ for $\lambda\in\Lambda^\ast$. If $\lambda\in \Xi^\ast\cap \ann(\lieh)$, the 
	Fourier coefficient on the left hand side is $1$, and likewise for the right hand side because $t^{-\lambda}$ restricts to the constant function $1$ on $\exp(\Xi+\h)$. If $\lambda\not\in \Xi^\ast\cap \ann(\lieh)$, both sides integrate to zero against $t^{-\lambda}$. 
\end{proof}

\end{appendix}

\bibliographystyle{amsplain}
\input{VerlindeSums3.bbl}


\end{document}

%% file: VerlindeSums3.bbl
\def\cprime{$'$} \def\polhk#1{\setbox0=\hbox{#1}{\ooalign{\hidewidth
  \lower1.5ex\hbox{`}\hidewidth\crcr\unhbox0}}} \def\cprime{$'$}
  \def\cprime{$'$} \def\cprime{$'$} \def\cprime{$'$} \def\cprime{$'$}
  \def\polhk#1{\setbox0=\hbox{#1}{\ooalign{\hidewidth
  \lower1.5ex\hbox{`}\hidewidth\crcr\unhbox0}}} \def\cprime{$'$}
  \def\cprime{$'$} \def\cprime{$'$} \def\cprime{$'$} \def\cprime{$'$}
\providecommand{\bysame}{\leavevmode\hbox to3em{\hrulefill}\thinspace}
\providecommand{\MR}{\relax\ifhmode\unskip\space\fi MR }
\providecommand{\MRhref}[2]{%
  \href{http://www.ams.org/mathscinet-getitem?mr=#1}{#2}
}
\providecommand{\href}[2]{#2}